\documentclass[11pt,a4paper]{article}

\usepackage{amsmath,amssymb,amsfonts,amsthm}
\usepackage{graphicx}
\usepackage{color}
\usepackage{subfig}
\usepackage{float}
\usepackage{multirow}
\usepackage{authblk}
\usepackage{algorithm}
\usepackage{algpseudocode}
\usepackage{booktabs}

\newtheorem{theorem}{Theorem}[section]% meant for sectionwise numbers
\newtheorem{proposition}[theorem]{Proposition}% 
\newtheorem{lemma}[theorem]{Lemma}%
\newtheorem{assumption}{Assumption}
\newtheorem{definition}{Definition}%

\DeclareMathOperator*{\diag}{diag}

\def\R{\mathbb{R}}
\def\N{\mathcal{N}}

\def\K{\mathcal{K}}
\def\dist{\textup{dist}}
\def\posto{\textup{rank}}
\def\0{\textbf{\textup{0}}}

\newcommand{\norm}[1]{\|#1\|_{2}} % norma 2
\begin{document}

\title{Levenberg-Marquardt method with Singular Scaling and applications}

\author[1]{Everton Boos\footnote{everton.boos@ufsc.br}}
\author[1]{Douglas S. Gon\c{c}alves\footnote{douglas.goncalves@ufsc.br}}
%\equalcont{These authors contributed equally to this work.}
\author[1]{Ferm\'{i}n S. V. Baz\'{a}n\footnote{fermin.bazan@ufsc.br}}
%\equalcont{These authors contributed equally to this work.}

\affil[1]{Department of Mathematics, Federal University of Santa Catarina, Florian\'{o}polis, 88040900, SC, Brazil}

\maketitle

\begin{abstract}
Inspired by certain regularization techniques for linear inverse problems, in this work we investigate the convergence properties of the Levenberg-Marquardt method using singular scaling matrices. Under a completeness condition, we show that the method is well-defined and establish its local quadratic convergence under an error bound assumption. We also prove that the search directions are gradient-related allowing us to show that limit points of the sequence generated by a line-search version of the method are stationary for the sum-of-squares function. 
%The usefulness of the method is illustrated with an application to parameter identification in a  heat conduction problem.
The usefulness of the method is illustrated with some examples of parameter identification in heat conduction problems
for which specific singular scaling matrices can be used to improve the quality of approximate solutions.
\end{abstract}


\section{Introduction}\label{sec:intro}
In this work we consider the following nonlinear least-squares problem 
\begin{equation}\label{problem}
\min_{x \in \R^n} \frac{1}{2}\|F(x)\|_2^2 =: \phi (x),
\end{equation}
where $F: \R^n \rightarrow \R^m$ is continuously differentiable. Throughout this paper we assume that the optimal value for \eqref{problem} is zero, i.e. the set 
\begin{equation*}
X^* := \{ x \in \R^n \mid F(x) = \0 \}
\end{equation*}
is non-empty. 

We shall provide local and global convergence analysis for the following Levenberg-Marquardt method (LMM)
\begin{align}\label{s1}
(J_k^TJ_k + \lambda_k L^TL) d_k &= -J_k^TF_k\\
x_{k+1} &= x_k + \alpha_k d_k, \quad \forall k \geq 0, \label{s2}
\end{align}
where $F_k:= F(x_k)$, $J_k:= J(x_k)$ is the Jacobian of $F$ at $x_k$, $\{\lambda_k \}$ is a 
sequence of positive scalars and $\{\alpha_k \}$ corresponds to a step-size sequence. 

In the classic LMM, instead of $L^T L$  a positive definite matrix $D_k$ 
called \emph{scaling matrix} is usually chosen. 
Convergence theory for this case and, in particular,  for the choice $D_k = I_n$, the identity matrix,
is well-known in the literature and local quadratic convergence
can be shown under mild assumptions
\cite{Dennis_Schnabel, Dennis1977, Yamashita, More78, Levenberg1944, Marquardt1963, Bergou2020, Behling2019}. 

In this work however, motivated by applications in inverse problems, we choose the scaling matrix in the form
$D_k = L^T L$ and allow it to be \emph{singular}.  
We shall refer to the iteration \eqref{s1}--\eqref{s2} as Levenberg-Marquardt method with Singular
Scaling (LMMSS).
%, for which we establish local and global convergence results under suitable assumptions and illustrate its effectiveness 
% in thermal conductivity reconstruction problems. 

%\tred{Paragrafo explicando metodos similares na literatura...}
Several authors over  the decades have proposed different choices for the scaling matrix, with
$D_k$  non-singular, showing advantages on the generated approximations made by LMM. 
Emblematic examples generally  take $D_k$ as diagonal \cite{Fletcher71}, for instance, 
Moré \cite{More78} introduces information from  the Jacobian $J_k$ on this   diagonal in the iterative process, aiming to help  
convergence and improve the conditioning of (\ref{s1}). Other definitions of $D_k$, which vary according to the context
of the problem,  may be  found in \cite{Dennis1977, Zhou98, Schwetlick89, Fan2003, Transtrum2012, Nowak2004, Osborne76}. 

On the other hand, our motivation to use singular matrices 
in the form $L^TL$ comes from the general form Tikhonov regularization
\cite{Tikhonov95}
for linear systems $Ax = b$, i.e.,
\begin{equation*}
	x_\lambda = (A^TA + \lambda L^TL)^{-1} A^Tb.
\end{equation*}
In this case, the literature indicates that for discrete ill-posed problems with smooth solution, the choice of
$L$ as a discrete version of derivative operators can lead to significant improvements in the generated approximations \cite{Hansen1998, Hansen2010}. 
%Hence, this paper studies the possibility of using these ideas in the iterative process of LMMSS. 

Hence, in this paper the use of a singular, and constant, $D_k = L^T L$ is not directly related to improving the conditioning of LMM subproblem or speeding-up convergence, but rather to allow us to introduce desired properties, e.g. smoothness, in the solution obtained by LMM in certain inverse problems. 
Our aim is to show that convergence of LMM with such singular scaling matrices can still be achieved under reasonable assumptions and that the use of certain problem-oriented singular scaling matrices has an impact in the quality of the obtained approximate solutions.

As far as we know, there are no works in the literature treating local and global convergence of LMM with singular scaling matrices. 
Although the roadmap of the convergence analysis follows closely the seminal work \cite{Yamashita}, there are new challenges imposed by the singularity of the scaling matrices in the proof of key lemmas that demands other theoretical tools and additional assumptions.

%\tred{This paper is organized as follows...}
%\tred{(Novo parágrafo deixando claro o propósito do artigo... Aproveitar algumas frases da carta-resposta...)}

This paper is organized as follows. In Section~\ref{sec:pre}, we start by presenting the mathematical background and preliminary results necessary to the upcoming sections. 
Under a completeness condition we show that the LMMSS iteration is well-defined and prove its local quadratic convergence under an error bound condition in Section~\ref{sec:local}. 
%Next, we prove that such convergence is quadratic and global convergence is verified if the step size is chosen by the Armijo condition. 
Next, we prove in Section~\ref{sec:global} that, for LMMSS with an Armijo line-search scheme, every limit point of the generated sequence is a stationary point of the sum-of-squares function. 
Finally, Section~\ref{sec:app} presents several examples of parameter identification on heat conduction problems 
%presents an application on reconstructing thermal conductivity
in which specific singular scaling matrices tend to be effective. 
The paper ends with final remarks and future perspectives in Section~\ref{sec:final}.

\section{Preliminaries and auxiliary results}\label{sec:pre}

In this section we present the assumptions needed in the convergence analysis and recall some facts about the generalized singular value decomposition that are useful to establish some auxiliary lemmas. 

Let $B(\bar{x},\rho)$ denotes a closed ball in $\R^n$ centered at $\bar{x}$ with radius $\rho$ and $\N (A)$ the null-space of a matrix $A$. 
Concerning the linear system \eqref{s1} and the matrix $L \in \R^{p \times n}$ we consider the following assumption. 
%\begin{assumption}\label{Hip_nullJL}
%	The matrix $L\in \R^{p\times n}$ is full-rank, where $m\geq n\geq p$, and 
%	\begin{equation}\label{kernelJL}
%	\tred{\N (J(x)) \cap \N (L) = \{ \0 \}, \quad  \forall x \in \R^n.}
%%	\N (J_k) \cap \N (L) = \{ \0 \}, \quad  k \geq 0.
%	\end{equation}
%	Furthermore, for every $x \in\R^n $ there exist $\rho>0$ and $\gamma >0$  such that 
%   \begin{equation}\label{comp-condition}
%      \norm{J(x) v}^2 + \norm{L v} ^2 \geq \gamma \norm{v}^2, \quad \forall v \in B(x,\rho).
%      %x^TJ(x)^T J(x)\,x+x^TL^TL\,x\ge \gamma \;x^Tx, \;\forall x\in B(\bar{x},\rho).
%	\end{equation}.
%\end{assumption}

\begin{assumption}\label{Hip_nullJL}
The matrix $L\in \R^{p\times n}$ is full-rank, where $m\geq n\geq p$, and there exist $\gamma >0$ 
	such that, for every $x \in\R^n $ 
   \begin{equation}\label{comp-condition}
      \norm{J(x) v}^2 + \norm{L v} ^2 \geq \gamma \norm{v}^2, \quad \forall v \in \R^n. %B(x,\rho).
	\end{equation}
\end{assumption}

Inequality \eqref{comp-condition} is related to the \textit{completeness condition}, well-known in the inverse problems literature (see \cite[p. 34]{Morozov1984} and  \cite[p. 197]{engl}). We remark that Assumption~\ref{Hip_nullJL} is equivalent to %A direct consequence of such assumption is that 
	\begin{equation}\label{kernelJL}
	\N (J(x)) \cap \N (L) = \{ \0 \}, \quad  \forall x \in \R^n.
	\end{equation}

Condition \eqref{kernelJL} allows us to show that, for $\lambda_k > 0$, $(J_k^TJ_k + \lambda_k L^TL)$ is symmetric positive definite and thus 
the linear system \eqref{s1} has a unique solution $d_k$. 
%Uma aplicação direta desta hipótese se encontra em garantir uma única solução ao sistema linear. De fato, veja que $(J_k^TJ_k + \lambda_k L^TL)$ é simétrica e, pela hipótese, é definida positiva. Mais especificamente, para $v \in \R^n \backslash \{ \0 \} $ qualquer, temos dois casos: se $v \in \N (L)^\perp$, então $Lv \neq \0$; se $v \in \N (L)$, então $J_k v \neq \0$, por (\ref{kernelJL}). Em ambos os casos, como $\R^n = \N (L) \oplus \N (L)^\perp$ e $\lambda_k >0$, segue que $v^T(J_k^TJ_k + \lambda_k L^TL)v > 0$. Concluímos portanto que a matriz inversa de $(J_k^TJ_k + \lambda_k L^TL)$ existe e, em particular, é simétrica e definida positiva. Logo, $d_k$ calculada através de (\ref{s1}) está bem definida. 
Moreover, $d_k$ coincides with the unique minimizer of the following quadratic model
\begin{equation}\label{theta}
\theta_k (d) =  \|J_kd + F_k\|_2^2 + \lambda_k\|Ld\|_2^2.
\end{equation}
%Notice that $d_k$ from \eqref{s1} is exactly the unique minimizer of $\theta_k (d)$ for $d\in \R^n$. 
%This follows from the fact that $J_k^T J_k + \lambda_k L^T L$ is positive definite thanks to condition \eqref{kernelJL}.

Motivated by the analysis in \cite{Yamashita}, we also assume that the Levenberg-Marquardt parameter (damping parameter) $\lambda_k$ is chosen as  follows.

\begin{assumption}\label{Hip_lambda}
	The damping parameter is chosen as $\lambda_k = \lambda (x_k) := \|F(x_k)\|_2^2$, for every $k\geq 0$. 
\end{assumption}

An useful tool for analyzing the matrix pair $(J_k,L)$ involved in LMMSS iterations is the generalized singular value decomposition (GSVD). 
It was introduced by Van Loan \cite{VanLoan76} and further developed in subsequent works \cite{VanLoan1985, Paige81}. 
Here we present it according to \cite[p. 22]{Hansen1998}. For a proof and other properties we recommend \cite{Bjorck1990Handbook}.

\begin{theorem}[{{GSVD, Hansen \cite[p. 22]{Hansen1998}}}]\label{teoGSVD}
	Consider a matrix pair $(A,L)$, where $A\in \R^{m\times n}$, $L\in \R^{p\times n}$, $m\geq n\geq p$, $\posto (L) = p$ and $\N (A) \cap \N(L) = \{ \0 \}$. 
	Then, there exist matrices $U \in\R^{m\times n}$ and $V\in \R^{p\times p}$ with orthonormal columns %(i.e., $U^{ T}U = I_n$ e $V^{ T}V = I_{p }$) 
	and a non-singular $X\in \R^{n\times n}$ such that 
	\begin{equation}\label{GSVD}
	A = U \left[ \begin{array}{cc}
	\Sigma & \0 \\ 
	\0 & I_{n-p}
	\end{array}  \right] X^{-1}  \qquad \text{and} \qquad L = V \left[ \begin{array}{cc}
	M & \0
	\end{array}  \right] X^{-1},
	\end{equation}
	where $\Sigma$ and $M$ are diagonal matrices:
	\begin{equation}
	\Sigma = \diag (\sigma_{1}, \dots, \sigma_{p}) \in \R^{p\times p} \quad \text{and} \quad M = \diag (\mu_{1}, \dots, \mu_{p}) \in \R^{p\times p}.
	\end{equation}
	Moreover, the diagonal elements of $\Sigma$ and $M$ are non-negative, ordered such that 
	\begin{equation}
	0 \leq \sigma_{1} \leq \dots \leq \sigma_{p} \leq 1 \quad \text{and} \quad 1\geq \mu_{1} \geq \dots \geq \mu_{p} >0
	\end{equation}
	and normalized by the relation $\sigma_{ i}^{ 2} +\mu_{ i}^{ 2} = 1$, for $i = 1, \dots, p$. 
	The quotients 
	\begin{equation*}
	\gamma_{ i} = \frac{\sigma_{ i}}{\mu_{ i}}, \quad i = 1, \dots, p,
	\end{equation*}
	are called \textup{generalized singular values of the pair} $(A,L)$.
\end{theorem}

%%No contexto deste trabalho, para cada $x\in \R^n$ fixo, considere a decomposição GSVD do par $(J(x),L)$, que pode ser escrita como
%%\begin{equation*}
%%J (x) = U (x) \left[ \begin{array}{cc}
%%\Sigma (x) & \0 \\ 
%%\0 & I_{n-p}
%%\end{array}  \right] X (x)^{-1}  \quad \text{e} \quad L = V (x) \left[ \begin{array}{cc}
%%M (x) & \0
%%\end{array}  \right] X (x)^{-1},
%%\end{equation*}
%%com $\Sigma (x)$ e $M (x)$ possuindo 
%%\begin{equation*}
%%0 \leq \sigma_{1} (x) \leq \dots \leq \sigma_{p} (x) \leq 1 \quad \text{e} \quad 1\geq \mu_{1} (x) \geq \dots \geq \mu_{p} (x) >0
%%\end{equation*}
%%como elementos de suas diagonais respectivas. No caso em que $x$ assume valores em uma sequência $\{x_k\}$, ou seja, $x=x_k$, para cada $k\geq 0$, é comum que simplifiquemos a notação como feito em (\ref{s1}), de modo que a GSVD do par $(J(x_k), L) = (J_k, L)$ assume a forma
%%\begin{equation*}
%%J_k = U_k \left[ \begin{array}{cc}
%%\Sigma_k & \0 \\ 
%%\0 & I_{n-p}
%%\end{array}  \right] X_k^{-1}  \quad \text{e} \quad L = V_k \left[ \begin{array}{cc}
%%M_k & \0
%%\end{array}  \right] X_k^{-1},
%%\end{equation*}
%%em que, como acima,  
%%\begin{equation*}
%%(\Sigma_k)_{ii} := \sigma_{i,k} \quad \text{e} \quad (M_k)_{ii} := \mu_{i,k}, \quad i = 1, \dots, p.
%%\end{equation*}
%%Para referência futura, com o devido abuso de notação, podemos utilizar as mesmas definições acima para quando $x = y_k$ ou $x = z_k$, a depender do contexto.

%\tred{[D: Por que não considerar a GSVD do par $(J_k, \sqrt{\lambda_k} L)$ ??]}

GSVD of the pair $(J_k,L)$ gives us a precise characterization of the direction $d_k$ in LMMSS.
In fact, given the GSVD 
\begin{equation*}
J_k = U_k \left[ \begin{array}{cc}
\Sigma_k & \0 \\ 
\0 & I_{n-p}
\end{array}  \right] X_k^{-1}  \quad \text{and} \quad L = V_k \left[ \begin{array}{cc}
M_k & \0
\end{array}  \right] X_k^{-1},
\end{equation*}
where 
\begin{equation*}
(\Sigma_k)_{ii} := \sigma_{i,k} \quad \text{and} \quad (M_k)_{ii} := \mu_{i,k}, \quad i = 1, \dots, p,
\end{equation*}
we obtain 
\begin{equation}\label{JmuL}
J_k^TJ_k + \lambda_k L^TL = X_k^{-T} \left[ \begin{array}{cc}
\Sigma_k^2 + \lambda_k M_k^2 & \0 \\ 
\0 & I_{n-p}
\end{array}  \right] X_k^{-1}.  %\quad \forall k \in \mathbb{N},
\end{equation}
Then, $d_k$ from \eqref{s1} can be expressed as 
\begin{equation*}
d_k = -X_k \left[ \begin{array}{cc}
(\Sigma_k^2 + \lambda_k M_k^2)^{-1}& \0 \\ 
\0 & I_{n-p}
\end{array}  \right] X_k^T\left (X_K^{-T}\left[ \begin{array}{cc}
\Sigma_k & \0 \\ 
\0 & I_{n-p}
\end{array}  \right] U_k^T \right) F_k, %\quad \forall k \in \mathbb{N},
\end{equation*}
and further reduced to 
\begin{equation}\label{GSVD_dk}
d_k =  -X_k \left[ \begin{array}{cc}
\Gamma_k & \0 \\ 
\0 & I_{n-p}
\end{array}  \right] U_k^T  F_k, \quad \text{where} \quad \Gamma_k := (\Sigma_k^2 + \lambda_k M_k^2)^{-1}\Sigma_k.
\end{equation}
Since the columns of $U_k$ are orthonormal, it follows that 
\begin{equation}\label{est_dk}
\|d_k\|_2 \leq \|X_k\|_2 \max \{ \|\Gamma_k \|_2, 1 \} \|F_k\|_2, \quad \forall k \geq 0.
\end{equation}
%uma relação de suma importância que usaremos adiante. De fato, por (\ref{est_dk}) aparecer frequentemente e seus termos envolverem $X_k$ e $\Gamma_k$, estimativas para as normas de ambas se tornam imperativas. Um primeiro resultado nessa direção
%vem no seguinte lema.
The next two lemmas give bounds for $\|X_k\|$ and $\|\Gamma_k\|$ that will be helpful later to bound $\| d_k\|$ through \eqref{est_dk}.

\begin{lemma}\label{estimXk} For every $k \in \mathbb{N}$, we have 
     \begin{equation} \label{cotaXk}
     \|X_k\|_2 \le \dfrac{1}{\sqrt{\gamma}}.
     \end{equation}
%     Alternativamente, 
%    \begin{equation}\label{cotaXk-1}
%	\|X_k\|_2 \leq \nu_p^{-1}, \quad \text{com} \quad \nu_p = \left\{  \begin{array}{ll}
%	\|L^{-1}\|_2, & p = n\\
%	\max \{ \|L^\dagger \|_2, \inf (J_kP_{\N (L)})^{-1} \}, & p < n
%	\end{array} \right. 
%	\end{equation} 
%	em que $P_{\N (L)}$ corresponde à matriz de projeção em $\N (L)$ e $\inf (J_kP_{\N (L)})$ denota o menor valor singular não nulo de $J_kP_{\N (L)}$.
\end{lemma}
\begin{proof}
From the GSVD of the pair $(J_k,L)$ it follows that  
\[ 
J_k^TJ_k + L^TL = X_k^{-T}X_k^{-1}, 
\]
as can be deduced from \eqref{JmuL} with $\lambda_k = 1$. 
Then, we obtain $X_kX_k^T = (J_k^TJ_k + L^TL)^{-1}$, and thus  
\[
\|X_k\|_2^2 = \| (J_k^TJ_k + L^TL)^{-1}\|_2. 
\]
Now, from Assumption~\ref{Hip_nullJL}, and Rayleigh quotient for symmetric matrices, we obtain inequality \eqref{cotaXk}. 
%Por outro lado, se denotamos 
%	\begin{equation*}
%	Z_k = \left[ \begin{array}{c}
%	J_k \\
%	L
%	\end{array} \right],
%	\end{equation*}
%então $$\|X_k\|_2^2 = \| (Z_k^TZ_k)^{-1}\|_2\Rightarrow \|X_k\|_2 = \|Z_k^\dag\|_2.$$
%Agora, basta usar  \cite[Theorem 2.3]{Hansen1989} e a estimativa  (\ref{cotaXk-1}) segue imediatamente.
\end{proof}

Next, let us derive a bound for $\|\Gamma_k\|_2$. 
From the GSVD of the pair $(J_k,L)$, with generalized singular values $\gamma_{ i,k} = \sigma_{ i,k}/\mu_{ i,k}$, 
we may write
\begin{equation*}
\sigma_{ i,k}^2 = \frac{\gamma_{ i,k}^2 }{\gamma_{ i,k}^2 + 1} \quad \text{and} \quad \mu_{ i,k}^2 = \frac{1}{ \gamma_{ i,k}^2 + 1},
\end{equation*}
which allows us to express the diagonal matrix $\Gamma_k$ from \eqref{GSVD_dk} as 
\begin{equation*}
(\Gamma_k)_{ii} = \frac{\sigma_{ i,k }}{ \sigma_{ i,k}^2 + \lambda_k \mu_{ i,k}^2} = \frac{ \gamma_{ i,k} }{\gamma_{ i,k}^2 + \lambda_k} \sqrt{\gamma_{ i,k}^2 + 1}, \quad i = 1, \dots, p.
\end{equation*}

%Retornando à equação (\ref{GSVD_dk}), temos a definição de
%\begin{equation*}
%\Gamma_k = (\Sigma_k^2 + \lambda_k M_k^2)^{-1}\Sigma_k, \quad \forall k \geq 0.
%\end{equation*}
%Agora, se denotamos os valores singulares generalizados $\gamma_{ i,k} = \sigma_{ i,k}/\mu_{ i,k}$, então é possível extrair relações como 
%\begin{equation*}
%\sigma_{ i,k}^2 = \frac{\gamma_{ i,k}^2 }{\gamma_{ i,k}^2 + 1} \quad \text{and} \quad \mu_{ i,k}^2 = \frac{1}{ \gamma_{ i,k}^2 + 1},
%\end{equation*}
%que o leitor pode obter diretamente de $\sigma_{ i,k}^2 + \mu_{ i,k}^2 = 1$ e da definição de $\gamma_{ i,k}$ (veja \cite[p. 494]{Hansen1989}, para maiores detalhes). Desta forma, podemos escrever $\Gamma_k$ como
%\begin{equation}\label{Gamma_k}
%\Gamma_k = \diag \left (\dfrac{\gamma_{1,k}}{(\gamma_{1,k}^2+\lambda_k)}\sqrt{1+\gamma_{1,k}^2},\dots, \dfrac{\gamma_{p,k}}{(\gamma_{p,k}^2+\lambda_k)}\sqrt{1+\gamma_{p,k}^2} \right ) \in \MD{\R}{p}{p},
%\end{equation}
%uma vez que
%\begin{equation*}
%(\Gamma_k)_{ii} = \frac{\sigma_{ i,k }}{ \sigma_{ i,k}^2 + \lambda_k \mu_{ i,k}^2} = \frac{ \gamma_{ i,k} }{\gamma_{ i,k}^2 + \lambda_k} \sqrt{\gamma_{ i,k}^2 + 1}, \quad i = 1, \dots, p.
%\end{equation*}
%Portanto, para produzir cotas para $\|\Gamma_k\|_2$ precisamos analisar o comportamento da relação acima com respeito a variações nos valores singulares generalizados e em $\lambda_k$, como faremos com o auxílio do seguinte lema.

Therefore, we can bound $\| \Gamma_k \|_2$ in terms of the generalized singular values $\gamma_{i,k}$ and the damping parameter $\lambda_k$. 
To this end, we need the following technical lemma, whose proof can be found in the appendix.

\begin{lemma}\label{lemma_psi}
	For the function
	\begin{equation*}
	\psi(\gamma,\lambda)=\dfrac{\gamma \sqrt{1+\gamma^2}}{\gamma^2+\lambda}, \quad
	\gamma \geq 0, \ \lambda > 0,
	\end{equation*}
	the following properties hold.
	\begin{itemize}
		%\item[(a)] Para $\lambda = 0$, a função $\psi (\gamma, \lambda)$ é decrescente em $\gamma$. Além disso, $\lim_{ \gamma \rightarrow 0} \psi (\gamma, \lambda) = \infty$, ou seja, $\psi$ é ilimitada.
		\item[(a)] For a fixed $\lambda \in (0,1/2)$, the function $\psi (\gamma, \lambda)$ has a unique maximum achieved at 
		\begin{equation*}
		\gamma_{\max} = \sqrt{ -\frac{ \lambda}{ 2\lambda-1}} \quad \text{with value} \quad \psi(\gamma_{\max},\lambda) = \max_{\gamma \geq 0} \psi(\gamma,\lambda)= \frac{1}{ 2 \sqrt{ \lambda - \lambda^2 }}.
		\end{equation*}
		\item[(b)] For a fixed $\lambda \in [1/2,+\infty)$, the function $\psi (\gamma, \lambda)$ is non-decreasing and bounded above. More precisely,
		\begin{equation*}
		\psi (\gamma, \lambda) \leq 1, \quad \forall \gamma \geq 0.
		\end{equation*}
	\end{itemize}
	%	\begin{equation*}
	%	\max_{\gamma > 0} \psi(\gamma,\lambda) \left\{ \begin{array}{ll}
	%	= \infty, & \lambda = 0\\
	%	= \dfrac{1}{ 2 \sqrt{ \lambda - \lambda^2 }}, & 0 < \lambda < 1/2\\
	%	\leq 1, & 1/2 \leq \lambda 
	%	\end{array} \right.
	%	\end{equation*}
\end{lemma}

\section{Local convergence}\label{sec:local}
In this section, influenced by the convergence analysis of Yamashita and Fukushima \cite{Yamashita} for LMM with $I_n$ as scaling matrix,  
we prove local convergence of LMMSS iterations \eqref{s1}--\eqref{s2}, with $\alpha_k=1$, under Assumptions~\ref{Hip_nullJL}, \ref{Hip_lambda} 
and an error bound condition.

%Em termos de organização, principiamos apresentando alguns conceitos e hipóteses consideradas, com foco na demonstração de dois teoremas centrais: o primeiro trata da convergência local quadrática propriamente dita, no contexto de (\ref{conv_quad}); o segundo une tal resultado com a convergência global, afirmando que se pontos limites de $\{ x_k \}$ são mínimos locais, então a partir de algum índice $k$, $x_k$ se aproxima quadraticamente da solução. Ambos necessitam de resultados auxiliares exibidos na forma de lemas.

Let us start by stating the additional assumptions required in our analysis. 

\begin{assumption}\label{Hip_J}
	For any $x^* \in X^*$, there exist constants $\delta \in (0,1/2)$ and $c_1 > 0$ such that 
	\begin{equation}\label{quadapprox}
	\|J(y)(x-y) - (F(x)-F(y))\|_2 \leq c_1 \|x-y\|_2^2,
	\end{equation}
	for all $x, y \in B(x^*, 2\delta)$. %:= \{x \in \R^n \ | \ \|x-x^*\|_2 \leq b \}$. 
\end{assumption}

Such assumption is standard in convergence analysis of LMM. 
We remark that if $F$ is continuous differentiable and $J$ is Lipschitz, then \eqref{quadapprox} holds. 
The left hand side of inequality \eqref{quadapprox} is often called \emph{linearization error}. 
%%De fato, $J$ ser Lipschitz implica na existência de uma constante $M_J >0$ tal que
%%\begin{equation*}
%%\norm{J(x) - J(y)} \leq M_J \norm{x-y}, \quad x,y \in B(x^*, 2\delta).
%%\end{equation*}
%%Neste caso, fazendo uso da aproximação de Taylor de primeira ordem para $x$ e $y$, possível por $F$ ser continuamente diferenciável, segue que
%%\begin{equation*}
%%F(x) - F(y) = \int_{0}^{1} J( x + t(y-x) )(x-y) dt.
%%\end{equation*}
%%Portanto,
%%\begin{align*}
%%\norm{J(y)(x-y) - (F(x)-F(y))} &= \left\| \int_{0}^{1} \left[J(y) - J( x + t(y-x) )\right](x-y) dt \right\|_2\\
%%&\leq \int_{0}^{1} \left\| J(y) - J( x + t(y-x) ) \right\|_2 \norm{x-y} dt\\
%%&\leq \int_{0}^{1} M_J \norm{(1-t) (y-x)} \norm{x-y} dt\\
%%&= M_J \norm{x-y}^2 \int_{0}^{1} (1-t) dt\\
%%&= \frac{M_J}{2} \norm{x-y}^2, \quad x,y \in B(x^*, 2\delta).
%%\end{align*}
%A relação (\ref{quadapprox}) é comumente chamada de \textit{erro da aproximação linear} e, para maiores detalhes na demonstração acima, veja Dennis e Schnabel \cite[Theorem 4.1.12]{Dennis_Schnabel}. 

Additionally, \eqref{quadapprox} implies that there exists a constant $c_F>0$ such that 
\begin{equation}\label{F_approx}
\|F(x)-F(y)\|_2 \leq c_F \|x-y\|_2, \quad \forall x, y \in B(x^*, 2\delta).
\end{equation}
%Para tanto, observe que por $F$ ser continuamente diferenciável e $B(x^*, 2\delta)$ ser um conjunto compacto, existe uma constante $\widehat{M}_J >0$ tal que $\norm{J(x)} \leq \widehat{M}_J$, para todo $x \in B(x^*, 2\delta)$, pelo teorema de Weierstrass. Assim, disto e usando (\ref{quadapprox}), temos que
%\begin{align*}
%\|F(x)-F(y)\|_2 &\leq \|F(x)-F(y) \pm J(y)(x-y)\|_2\\
%&\leq \norm{J(y)(x-y) - (F(x)-F(y))} + \norm{J(y)} \norm{x-y}\\
%&\leq c_1 \|x-y\|_2^2 + \widehat{M}_J\norm{x-y}\\
%&\leq c_F \norm{x-y},\quad \forall x, y \in B(x^*, 2\delta),
%\end{align*}
%para $c_F := c_1 + \widehat{M}_J$ e considerando, claro, que os pontos estão suficientemente próximos a ponto de $\|x-y\|_2^2 \leq \|x-y\|_2$, o que é natural dada a caracterização local dos resultados aqui procurados.

Last, we need the following error bound condition. 

\begin{assumption}[\textit{Error bound}]\label{Hip_error_bound}
	For any $x^* \in X^*$, $\|F(x)\|_2$ provides a local error bound in $B(x^*, 2\delta)$ for the system of nonlinear equations $F(x) = \0$, i.e., 
	there exists $c_2 \in (0, \infty)$ such that 
	\begin{equation}\label{error_bound}
	c_2 \dist (x, X^*) \leq \|F(x)\|, \quad \forall x \in B(x^*, 2\delta).
	\end{equation}
\end{assumption}

%Condições \textit{error bound} vem sendo utilizadas na literatura em muitos trabalhos, especialmente nos últimos 20 anos, apesar de que os primeiros artigos no assunto sejam ainda dos anos 1950 \cite{Hoffman1952,Robinson81,Behling2019,Yamashita,FanYuan2005,Kanzow2004}. Em geral, se tratam de condições de regularidade que permitem o tratamento/análise de soluções não isoladas (mas não somente) \cite{Behling2019,Behling2021}. Costumam também ser atrativas por exigirem menos que algumas das hipóteses de regularidade usuais, como a posto completo (ou ainda não singularidade) da matriz Jacobiana na solução. Para LMM clássico, por exemplo, esta é uma condição comum em demonstrações da convergência, como feito em \cite{Dennis_Schnabel,More78}. 

Error bound conditions have been widely used in the literature \cite{Hoffman1952,Robinson81,Behling2019,Yamashita,FanYuan2005,Kanzow2004} 
as regularity conditions that allow us to analyze local convergence to non-isolated solutions \cite{Behling2019,Behling2021}. 
They are attractive for the convergence analysis of LMM for being weaker than other regularity conditions such as full-rank of the Jacobian. 

In what follows, we shall use the notation $\bar{x}_k$ to denote an element of $X^*$ such that 
\begin{equation*}
\|x_k -\bar{x}_k\|_2 = \dist(x_k, X^*).
\end{equation*}
%Além disso, assumimos que $x^*$ é como nas Hipóteses \ref{Hip_J} e \ref{Hip_error_bound} e, adicionalmente, são válidas durante toda a seção as Hipóteses \ref{Hip_nullJL}-\ref{Hip_lambda}, permitindo que os coeficientes $\lambda_k$ estejam bem determinados e que tenhamos propriedades acerca das matrizes $J_k$ e $L$. Finalmente, é importante ressaltar que esta subseção se foca em verificar convergência local de iterações LMMSS sem escolha de passo, isto é, calculamos $d_k$ como em (\ref{s1}) e atualizamos
We also recall that in the local convergence analysis, the step-size is fixed $\alpha_k=1$ through the iterations. 
Thus the update \eqref{s2} reduces to
\begin{equation}\label{xk_passo_unitario}
x_{k+1} = x_k+d_k, \quad k\geq 0.
\end{equation}
%A questão de escolha de passo, que faremos através do critério de Armijo, retorna na subseção seguinte, em que faremos uma análise global da técnica. Note que a escolha de passo por Armijo garante que pontos limite da sequência gerada por LMMSS são estacionários. Agora, se tal ponto limite for um minimizador do problema, a ideia é verificar que a partir de um certo índice $k$ passos de tamanho 1 são sempre tomados, permitindo portanto a aliança entre a análise local e a global (com Armijo).

The next lemma is a generalization of a key lemma from \cite{Yamashita} to the context of LMMSS.

\begin{lemma}\label{lemma_1}
	Let Assumptions \ref{Hip_nullJL}--\ref{Hip_error_bound} hold. % and $L \in \R^{p \times n}$, with $\text{rank}(L) = p$, be a scaling matrix. 
	If $x_k \in B(x^*, \delta)$, then $d_k$ from \eqref{s1} satisfies 
	\begin{equation}\label{eq:dk_dist}
	\|d_k\|_2 \leq c_3 \dist (x_k, X^*),
	\end{equation}
	and
	\begin{equation}\label{eq:JdFk_dist}
	\|J_k d_k + F_k\|_2 \leq c_4 \dist (x_k,X^*)^2,
	\end{equation}
	where $c_3 = \dfrac{1}{c_2 \sqrt{\gamma}} \sqrt{ c_1^2 + c_2^2 (\| L \|_2^2+ c_F^2) }$ e $c_4 = \sqrt{c_1^2 + c_F^2 \| L \|_2^2}$.
\end{lemma}
\begin{proof}
	Since $x_k \in B(x^*, \delta)$, it follows that 
	\begin{equation*}
	\|\bar{x}_k - x^*\|_2 \leq \|\bar{x}_k - x_k\|_2 + \|x_k - x^*\|_2 \leq \|x_k - x^*\|_2 + \|x_k - x^*\|_2 \leq 2 \delta,
	\end{equation*}
	which implies in $\bar{x}_k \in B(x^*, 2 \delta)$. 
	
	Then, from Assumption~\ref{Hip_nullJL}, with $v=d_k$, we obtain 
	\begin{equation}\label{estim-dk}
	 \gamma \|d_k\|_2^2 \le \|J_kd_k\|_2^2+\|Ld_k\|_2^2.
	 \end{equation}
	Let us bound each term in the right hand side separately. First, 
%	\begin{equation*}
%	\|J_k(J_k^TJ_k+\lambda_kL^T L)^{-1}J_k^{T}\|_2\le 1,
%	\end{equation*}
	from the GSVD of the pair $(J_k,L)$, it follows that 
	\begin{equation*}
	J_k(J_k^TJ_k+\lambda_kL^T L)^{-1}J_k^{T} = U_k \left[ \begin{array}{cc}
	\Sigma_k (\Sigma_k^2 + \lambda_k M_k^2)^{-1}\Sigma_k & \0 \\ 
	\0 & I_{n-p}
	\end{array}  \right] U_k^T,
	\end{equation*}
	and since the columns of $U_k$ are orthonormal, we have 
	\begin{equation*}
	\|J_k(J_k^TJ_k+\lambda_kL^T L)^{-1}J_k^{T}\|_2 \leq \max_{ 1\leq i \leq p} \left\{ \frac{\sigma_{ i,k }^2 }{ \sigma_{ i,k }^2 + \lambda_k \mu_{ i,k}^2}, 1 \right\} \leq 1,
	\end{equation*}
	because $\lambda_k \mu_{ i,k}^2 > 0$, for every $k$. From this inequality and \eqref{s1}, it follows that 
	\begin{equation}\label{ineq1} 
	\|J_kd_k\|_2^2 = \|J_k(J_k^TJ_k+\lambda_kL^T L)^{-1}J_k^{T}F_k\|_2^2\le 
	\|F_k\|_2^2\le c_F^2\|\bar x_k-x_k\|_2^2,
	\end{equation}
	where the last inequality came from \eqref{F_approx} and the fact that $F(\bar x_k) = \0$. 
	
	Second, since $d_k$ is the unique minimizer of $\theta_k(d)$ (see~\eqref{theta}), we have 
	 $\lambda_k\|Ld_k\|_2^2 \le \theta_k(d_k)\le \theta_k(d)$, for all $d \in \R^n$, in particular, for $d = \bar{x}_k - x_k$:
	\begin{equation*}
	\lambda_k \|Ld_k\|_2^2 \le \theta_k(d_k) \le \theta_k(\bar x_k-x_k) = \|J_k(\bar x_k-x_k)+ F_k\|_2^2+ 
	\lambda_k\|L(\bar x_k-x_k)\|_2^2.
	\end{equation*}
Then, from Assumption~\ref{Hip_J}, % and for $\norm{L} = 1$, 
we obtain 
\[
	 \lambda_k \|Ld_k\|_2^2 \le c_1^2\|\bar x_k-x_k\|_2^4+\lambda_k \| L \|_2^2 \|\bar x_k-x_k\|_2^2.
\]
	%Por outro lado, pela Hipótese \ref{Hip_lambda} e a condição de \textit{error bound}, temos que $\lambda_k = \|F_k\|_2^2\ge c_2^2\|\bar x_k-x_k\|_2^2$. Portanto, com base no disposto acima e nesta desigualdade,
By using Assumptions~\ref{Hip_lambda} and \ref{Hip_error_bound}, we obtain
\[
\lambda_k = \|F_k\|_2^2\ge c_2^2\|\bar x_k-x_k\|_2^2,
\]
and plugging it in the previous inequality leads to 
	%, a Hipótese \ref{Hip_lambda}, e a condição de \textit{error bound} $\|F_k\|^2\ge c_2^2\|\bar x_k-x_k\|^2$, segue que
	\begin{equation}\label{ineq2}
	\|L d_k\|_2^2\leq \frac{c_1^2 }{ \lambda_k }\|\bar x_k-x_k\|_2^4+ \| L \|_2^2 \|\bar x_k-x_k\|_2^2 \le \dfrac{c_1^2+c_2^2 \|L\|_2^2 }{c_2^2}\|\bar x_k-x_k\|_2^2.
	\end{equation}
%	em que utilizamos a Hipótese \ref{Hip_lambda} e a condição de \textit{error bound} para produzir . 
	
	Then, replacing \eqref{ineq1} and \eqref{ineq2} in \eqref{estim-dk}, implies that 
	\begin{equation*}
	\|d_k\|_2 \le c_3 \|\bar x_k-x_k\|_2,
	\end{equation*}
	with $c_3 = \dfrac{1}{c_2 \sqrt{\gamma}} \sqrt{ c_1^2 + c_2^2 (\| L \|_2^2+ c_F^2) }$.

	To prove \eqref{eq:JdFk_dist}, we use again the fact that $d_k$ is the minimizer of $\theta_k (d)$, that $F(\bar{x}_k) = \0$ 
	and Assumption~\ref{Hip_J} to obtain 
	%A segunda parte do lema é feita como a seguir, não dependendo da estrutura acima desenvolvida. De fato \reminder{detalhar melhor}:
	\begin{align*}
	\|J_kd_k + F_k\|_2^2 &\leq \theta_k (d_k) \leq \theta_k (\bar{x}_k - x_k) \\
	&= \|J_k(\bar{x}_k - x_k) - (F(\bar{x}_k) - F_k)\|_2^2 + \lambda_k\|L(\bar{x}_k - x_k)\|_2^2\\
	&\leq c_1^2\|\bar{x}_k - x_k\|_2^4 + \lambda_k \| L \|_2^2 \|\bar{x}_k - x_k\|_2^2.
	\end{align*}
 Since $\lambda_k = \|F_k\|_2^2 = \|F_k - F(\bar{x}_k)\|_2^2 \leq c_F^2 \|x_k - \bar{x}_k\|_2^2$, we conclude that 
	\begin{equation*}
	\|J_kd_k + F_k\|_2^2 \leq (c_1^2 + c_F^2 \|L\|_2^2) \|x_k - \bar{x}_k\|_2^4.
	\end{equation*}
	Thus, defining $c_4 = \sqrt{c_1^2 + c_F^2 \|L\|_2^2}$, we have proved 
	\begin{equation*}
	\|J_kd_k + F_k\|_2 \leq c_4 \|x_k - \bar{x}_k\|_2^2 = c_4 \dist (x_k,X^*)^2.
	\end{equation*}
\end{proof}

With Lemma~\ref{lemma_1} at hand, the remainder of the local convergence analysis is exactly the same as in the seminal work \cite{Yamashita}. 
Hence, we only recap the main results in \cite{Yamashita} without proofs. In summary, these results state that if $x_0$ is sufficiently close to some $x^* \in X^*$, 
then the sequence $\{x_k\}$ generated by LMMSS (with $\alpha_k=1$) converges to some point in $X^*$ quadratically. 

Next lemma is useful to show that $\dist (x_k,X^*)$ goes to zero quadratically. 

\begin{lemma}[{{Yamashita e Fukushima \cite{Yamashita}}}]\label{lemma_2}
	If $x_k, x_{k-1} \in B(x^*, \delta)$, then it holds that 
	\begin{equation*}
	\dist (x_k,X^*) \leq c_5 \dist (x_{k-1},X^*)^2,
	\end{equation*}
	where $c_5 = (c_1 c_3^2 + c_4)/c_2$.
\end{lemma}

%\begin{proof}
%	Como $x_k, x_{k-1}\in B(x^*, \delta)$ e $x_k = x_{k-1} + d_{k-1}$, segue da Hipótese \ref{Hip_error_bound} que % (a) e (b) e do Lema \ref{lemma_1} que
%	\begin{align*}
%	c_2 \dist (x_k,X^*) &= c_2 \dist (x_{k-1} + d_{k-1},X^*)\\
%	&\leq \|F(x_{k-1} + d_{k-1})\|_2\\
%	&= \|F(x_{k-1} + d_{k-1}) \pm J_{k-1}d_{k-1} \pm F(x_{k-1})\|_2\\
%	&\leq \norm{ J_{k-1}d_{k-1} - (F(x_{k-1} + d_{k-1}) - F(x_{k-1})) } + \norm{J_{k-1}d_{k-1} + F(x_{k-1})} \\
%	&\leq c_1 \|d_{k-1}\|_2^2 + \|J_{k-1}d_{k-1} + F(x_{k-1})\|_2,
%	%&\leq \|J_{k-1}d_{k-1} + F(x_{k-1})\|_2 + c_1 \|d_{k-1}\|_2^2\\
%	%&\leq c_4 \dist (x_{k-1},X^*)^2 + c_1 c_3^2\dist (x_{k-1},X^*)^2 \qquad \text{(Lema \ref{lemma_1})}\\
%	%&= (c_1 c_3^2 + c_4) \dist (x_{k-1},X^*)^2.
%	\end{align*}
%	a última linha um resultado de aplicar a Hipótese \ref{Hip_J} para $x := x_{k-1} + d_{k-1}$ e $y := x_{k-1}$. Agora, utilizando as desigualdades obtidas no Lema \ref{lemma_1}, temos
%	\begin{equation*}
%	c_2 \dist (x_k,X^*) \leq c_1 c_3^2\dist (x_{k-1},X^*)^2 + c_4 \dist (x_{k-1},X^*)^2.
%	\end{equation*}
%	Basta então definir $c_5 = (c_1 c_3^2 + c_4)/2$ para concluir.
%\end{proof}

Now, Lemma~\ref{lemma_3} ensures that if $x_0$ is close enough to $x^*$, then all the iterates $x_k$ also belong to a neighborhood of $x^*$ as well. 
%Este é o foco do próximo resultado.
%Now, we must ensure that if the initial point $x_0$ is sufficiently close to $x^*$ then $x_k \in B(x^*, b/2)$, for all $k$, which is proved in the next lemma.

\begin{lemma}[{{Yamashita e Fukushima \cite{Yamashita}}}]\label{lemma_3}
	If $x_0 \in B(x^*, r)$, then $x_k \in B(x^*, \delta)$ for every $k$, where 
	\begin{equation*}
	r := \min \left\{ \frac{\delta}{2+4c_3}, \frac{1}{2c_5} \right\}.
	\end{equation*}
\end{lemma}

%Com estes dois lemas, podemos finalmente provar o teorema de convergência local quadrática, a seguir. Reforçamos que as Hipóteses \ref{Hip_nullJL}--\ref{Hip_error_bound} seguem válidas.
%With these two last lemmas, we can prove the main convergence theorem.
With the last two lemmas, and considering Assumptions \ref{Hip_nullJL}--\ref{Hip_error_bound}, we have the following theorem about the quadratic local convergence of LMMSS.

\begin{theorem}\label{theorem_1}
	%Suponha que as hipóteses (a), (b) e (c) valem. 
	Let $\{ x_k \}$ be the sequence generated by LMMSS: 
	\begin{align*}
	(J_k^TJ_k + \lambda_k L^TL) d_k &= -J_k^TF_k, \\
	x_{k+1} &= x_k + d_k, 
	\end{align*}
	with $x_0 \in B(x^*, r)$. Then, $\{ \dist (x_k,X^*) \}$ converges to zero quadratically. 
	Furthermore, the sequence $\{ x_k \}$ converges to a solution $\hat{x} \in X^* \cap B(x^*, \delta)$.
\end{theorem}
\begin{proof}
The proof is the same as \cite[Theorem~2.1]{Yamashita} but using the inequalities and corresponding constants of Lemma~\ref{lemma_1}. 
\end{proof}

%\begin{proof}
%	A primeira parte do teorema segue diretamente dos Lemas \ref{lemma_2} e \ref{lemma_3}. Precisamos apenas mostrar que a sequência $\{ x_k \}$ converge para uma solução $\hat{x} \in X^* \cap B(x^*, \delta)$. Como $\{ \dist (x_k,X^*) \}$ converge para zero e $x_k \in B(x^*, \delta)$ para todo $k$, é suficiente mostrar que $\{ x_k \}$ converge. Como (\ref{l3_norm_d}) implica que
%	\begin{equation*}
%	\|d_k\|_2 \leq c_3 r \left( \frac{1}{2} \right)^{2k -1}
%	\end{equation*}
%	para todo $k\geq 1$, temos que para inteiros positivos $p, q$ tais que $p \geq q$,
%	\begin{equation*}
%	\|x_p - x_q\|_2 \leq \sum_{i = q}^{p-1} \|d_i\|_2 \leq \sum_{i = q}^{\infty} \|d_i\|_2 \leq c_3 r \sum_{i = q}^{\infty} \left( \frac{1}{2} \right)^{2i -1} = \frac{1}{3} c_3 r \left( \frac{1}{2} \right)^{2q -3},
%	\end{equation*}
%	com a série geométrica saindo de raciocínio similar a (\ref{l3_sum2}). Portanto, $\{ x_k \}$ é uma sequência de Cauchy e, portanto, é convergente (em $\R$).
%\end{proof}

\section{Global convergence}\label{sec:global}
\begin{algorithm}
	\caption{LMMSS with line-search}
	\label{alg:LMM_armijo}
	\begin{algorithmic}[1]
		\Require $\nu, \eta, \vartheta \in (0,1)$, $F$, $J$, $L$ and $x_0 \in \R^n$
		%\Ensure Ponto estacionário (possivelmente mínimo local) $\bar{x}$ para $\phi(x) = \frac{1}{2} \norm{F(x)}^2$
%		\State $k = 0$
		\State $\lambda_0 = \|F(x_0) \|_2^2$
		\For{$k = 0,1,\dots$}
		\If{(stopping criterion verified)}
		\State \Return $\bar{x} = x_{k}$ 
		\EndIf
		\State $d_k = -(J_k^T J_k + \lambda_k L^T L)^{-1} J_k^T F_k$ \Comment{Descent direction}
		\If{$\|F (x_k+ d_k)\|_2 \leq \vartheta \|F (x_k)\|_2$}
		\State $\alpha_k = 1$
		\Else \Comment{Armijo condition}
		\State Choose $m$ as the smallest non-negative integer such that 
		\begin{equation}\label{armijo}
		\phi (x_k+ \eta^m{d}_k) - \phi ({x}_k) \leq \nu \eta^m \nabla \phi ({x}_k)^T {d}_k
		\end{equation} 
		\State $\alpha_k = \eta^m$
		\EndIf
		\State ${x}_{k+1} = {x}_k+ \alpha_k {d}_k$
		\State $\lambda_{k+1} = \|F({x}_{k+1}) \|_2^2$
%		\State $k = k+1$
		\EndFor
	\end{algorithmic}
\end{algorithm}

In this section, we consider Algorithm~\ref{alg:LMM_armijo}: a version of LMMSS where the step-size is selected by line-search, such that the Armijo condition is satisfied, and prove that any limit point of the sequence generated by this algorithm is a stationary point for \eqref{problem}, regardless the starting point. 

To this end, it will suffice to prove that the sequence of directions $\{d_k\}$ generated by Algorithm~\ref{alg:LMM_armijo} is \emph{gradient-related} 
and then use a result from \cite[Proposition~1.2.1]{Bertsekas1999}.  

\begin{definition}[\textit{Gradient-related}] \label{def_grad_related}
	Let $\{x_k \}$ and $\{d_k\}$ be sequences in $\R^n$. 
	The sequence $\{d_k\}$ is said \textsf{\textup{gradient-related} to $\{ x_k \}$} if for every subsequence $\{x_k\}_{k\in \mathcal{K}}$ (with $\K \subseteq \mathbb{N}$) 
	converging to a non-stationary point, the corresponding subsequence $\{d_k\}_{k\in \mathcal{K}}$ is bounded and satisfies 
	\begin{equation}\label{liminf_dk}
	\limsup_{k \rightarrow \infty, \ k\in \mathcal{K}} \nabla \phi (x_k)^T d_k < 0.
	\end{equation}
\end{definition}

\begin{proposition}\label{prop_dk}
	%Sejam $F:\R^n \longrightarrow \R^m$, e $\phi (x)$ duas vezes continuamente diferenciável em um conjunto aberto convexo $D \subset \R^n$. Assuma que $J(x)$ é Lipschitz contínua em $D$ com constante $\gamma$, i.e.,
	%\begin{equation*}
	%\|J(x)-J(y)\|_2 \leq \gamma \|x-y\|_2, \quad \forall x,y \in D.
	%\end{equation*}
	%com $\|J(x)\|_2 \leq \alpha$, para todo $x \in D$. 
	Assume Assumptions~\ref{Hip_nullJL}~and~\ref{Hip_lambda} hold. Let $\{d_k\}$, $\{x_k\}$ be the sequences generated by Algorithm~\ref{alg:LMM_armijo}.  Then, $\{d_k\}$ is \textup{gradient-related}. %Seja $\{y_k\}$ uma sequência em $\R^n$ com uma subsequência $\{y_k\}_{k \in \mathcal{K }}$, $\mathcal{K} \subseteq \mathbb{N}$, que converge a um ponto $y_\infty$ não estacionário para $\phi$. Então, a sequência $\{d_k\}_{k\in \mathcal{K}}$, $d_k = d(y_k)$, gerada através de (\ref{s1}) é \textup{gradient-related}.
\end{proposition}

%\subsection{Estimating $d_k$}
%Usando (2) e GSVD: 
%\begin{equation*}
%d_k = -X_k \left[ \begin{array}{cc}
%(\Sigma_k^2 + \lambda_k M_k^2)^{-1}& 0 \\ 
%0 & I_{n-p}
%\end{array}  \right] X_k^T\left (X_K^{-T}\left[ \begin{array}{cc}
%\Sigma_k & 0 \\ 
%0 & I_{n-p}
%\end{array}  \right] U_k^T \right) F_k, \quad \forall k \in \mathbb{N},
%\end{equation*}
%reduces to
%\begin{equation*}
%d_k =  -X_k \left[ \begin{array}{cc}
%(\Sigma_k^2 + \lambda_k M_k^2)^{-1}\Sigma_k & 0 \\ 
%0 & I_{n-p}
%\end{array}  \right] U_k^T  F_k,
%\end{equation*}
%and this  can be written as
%$$
%d_k=-X_k \left[ \begin{array}{cc}\Gamma & 0 \\   0 & I_{n-p}\\
%\end{array}\right ]U_k^TF_k,$$
%in which $$\Gamma = {\rm diag}\left (\dfrac{\gamma_{1,k}}{(\gamma_{1,k}^2+\lambda_k)}\sqrt{1+\gamma_{1,k}^2},\dots, \dfrac{\gamma_{p,k}}{(\gamma_{p,k}^2+\lambda_k)}\sqrt{1+\gamma_{p,k}^2} \right ),$$
%where we have used (6)-(7) and the definition of $\gamma_{i,k}$. Now define
%$$\phi(\gamma,\lambda_k)=\dfrac{\gamma\sqrt{1+\gamma^2}}{(\gamma^2+\lambda_k)},\; 
%\gamma\ge 0, \;\lambda_k>0.$$
%Letting $\phi_{\max} = \displaystyle\max_{\gamma\ge 0}(\phi(\gamma,\lambda_k) ,$ it follows that,
%\begin{equation}\label{maxg}
%\|d_k\|_2 \le \|X\|_2 \max \{ \phi_{\max},1\}\|F_k\|_2.
%\end{equation}
%
%It remains to find $\phi_{\max}! $ as function of $\lambda_k=$ (and relate to $\|F_k\|_2$). 
%Bounds on $\|X\|_2$ could go as before!!!
\begin{proof}
	%Seja $\{x_k\}_{k\in \mathcal{K}}$, no conjunto de índices $\mathcal{K} \subseteq \mathbb{N}$, uma subsequência de $\{x_k\}$ que converge para um ponto não estacionário $x^*$. 
	Assume Algorithm~\ref{alg:LMM_armijo} generates an infinite sequence $\{x_k\}$, that is $\| F(x_k) \|_2 > 0$ for every $k$, and let $\{x_k\}_{k \in \mathcal{K }}$, $\mathcal{K} \subseteq \mathbb{N}$, be a subsequence converging to $x_{\infty}$, a non-stationary point for $\phi$, that is, $\nabla \phi (x_\infty) \neq \0$. 
%	Segundo a Definição \ref{def_grad_related}, precisamos verificar que $\{d_k\}_{k\in \mathcal{K}}$, $d_k = d(y_k)$, é limitada e satisfaz (\ref{liminf_dk}). 
	To soften the notation, from here on consider $k\in \mathcal{K}$. 
	From \eqref{JmuL} and the definition of $d_k$, it follows that 
\begin{align*}
\nabla \phi (x_k)^T d_k & = - d_k^T(J_k^T J_k + \lambda_k L^T L) d_k \\
\ & = -d_k^T X_k^{-T} \left[ \begin{array}{cc}
\Sigma_k^2 + \lambda_k M_k^2 & \0 \\ 
\0 & I_{n-p} \end{array}  \right] X_k^{-1} d_k \\
\ & :=  -d_k^T X_k^{-T} \Lambda_k X_k^{-1} d_k. \\
\end{align*}
From Ostrowski's law of inertia \cite[Theorem~1]{Ostrowski1959} we have 
\[
\lambda_n ( X_k^{-T} \Lambda_k X_k^{-1} ) = \sigma_n^2(X_k^{-1}) \lambda_n (\Lambda_k),
\]
where, with acknowledged abuse of notation, $\lambda_n(B)$ and $\sigma_n(B)$ stand for the smallest eigenvalue and smallest singular value of a matrix $B$, respectively. 
Then, from Lemma~\ref{estimXk}, it follows that 
\[
\lambda_n ( X_k^{-T} \Lambda_k X_k^{-1} ) \geq \gamma \lambda_n (\Lambda_k) \geq \gamma \min \{ 1, \lambda_k \},
\]
where we have used in the last inequality the fact that $\sigma_{i,k}^2 + \mu_{i,k}^2 = 1$. 
Since, $x_{\infty}$ is non-stationary, there exists a $\underline{\lambda} > 0$ such that for $k \in {\cal K}$, $\lambda_k = \|F(x_k)\|_2^2 \geq \underline{\lambda}$. 
Thus, we obtain
\begin{equation}\label{eq:angle}
\nabla \phi (x_k)^T d_k \leq - \gamma \min \{1,\underline{\lambda} \} \|d_k\|_2^2. 
\end{equation}

On the other hand, 
\begin{align*}
\| \nabla \phi(x_k) \|_2 & = \| (J_k^T J_k + \lambda_k L^T L) d_k \|_2 \leq \| J_k^T J_k + \lambda_k L^T L \|_2 \| d_k \|_2 \\
\ & \leq \left( \|J_k\|_2^2 + \| F_k \|_2^2 \| L \|_2^2 \right) \| d_k \|_2. \\
\end{align*}
Since $\{ \|F(x_k)\|_2 \}$ from Algorithm~\ref{alg:LMM_armijo} is non-increasing, we have $\|F(x_k)\|_2 \leq \| F(x_0)\|_2$, for every $k$. 
From the continuity of $J(x)$ and convergence of the subsequence $\{x_k\}_{\cal K}$, there exists $M_J > 0$ such that 
$\| J_k \|_2 \leq M_J$, for $k \in {\cal K}$. Therefore, 
\begin{equation}\label{eq:prop}
\| \nabla \phi(x_k) \|_2 \leq \left( M_J^2 + \| F(x_0)\|_2^2 \| L \|_2^2 \right) \| d_k \|_2 =: M_0 \|d_k\|_2, \quad \forall k \in {\cal K}.
\end{equation}

Hence, from \eqref{eq:prop} and \eqref{eq:angle}, we obtain
\[
\nabla \phi (x_k)^T d_k \leq - \frac{\gamma \min \{1,\underline{\lambda} \}}{M_0^2} \| \nabla \phi (x_k) \|_2^2, \quad \forall k \in {\cal K},
\]
which implies that 
\[
	\limsup_{k \rightarrow \infty, \ k\in \mathcal{K}} \nabla \phi (x_k)^T d_k \leq - \frac{\gamma \min \{1,\underline{\lambda} \}}{M_0^2} \| \nabla \phi (x_{\infty}) \|_2^2 < 0.
\]
	
It remains to prove that $\{d_k\}_{k\in \mathcal{K}}$ is bounded.  Recall from \eqref{est_dk} that 
	\begin{equation}\label{est_prop_dk}
	\|d_k\|_2 \leq \|X_k\|_2 \max \{ \|\Gamma_k \|_2, 1 \} \|F_k\|_2, \quad \forall k\geq 0,
	\end{equation}
	where $\Gamma_k = (\Sigma_k^2 + \lambda_k M_k^2)^{-1}\Sigma_k$.
	%Os próximos argumentos visam construir um limitante superior para (\ref{est_prop_dk}). %De agora em diante, considere $k\in \mathcal{K}$ e escolha $\varepsilon>0$ arbitrário. 
	From Lemma~\ref{estimXk}, we have $\|X_k\|_2 \leq 1/\sqrt{\gamma}$. 
%	Como $\{y_k\}_{k\in \mathcal{K}} \rightarrow y_\infty$, que é não estacionário, pela Proposição \ref{Hip_Xk} existem $c_X >0$ e $k_0 \geq 0$ tais que
%	\begin{equation}\label{est_prop_Xk}
%	\|X_k\|_2 \leq c_X, \quad \forall k \geq k_0. %\quad \text{e} \quad \|\Gamma_k\|_2 \leq c_\Gamma, \quad \forall k \geq k_0,
%	\end{equation}
	Let us analyze the term $\max \{ \|\Gamma_k \|_2, 1 \} \|F_k\|_2$. From Lemma~\ref{lemma_psi}, it follows that 
	\begin{equation}\label{norm_Gamma_k}
	\|\Gamma_k\|_2 \leq \dfrac{1}{ 2 \sqrt{ \lambda_k - \lambda_k^2 }}, \quad \text{if} \quad 0 < \lambda_k < \frac{1}{2}, \quad \text{and} \quad \|\Gamma_k\|_2 \leq 1, \quad \text{if} \quad \lambda_k \geq \frac{1}{2}.
	%\|\Gamma_k\|_2 = \left\{ \begin{array}{ll}
	%\dfrac{ \sqrt{1+\gamma_{1,k}^2}}{\gamma_{1,k}}, & \text{se } \lambda_k = 0\\
	%\dfrac{1}{ 2 \sqrt{ \lambda_k - \lambda_k^2 }}, & \text{se } 0 < \lambda_k < 1/2
	%\end{array} \right. \quad \text{e} \quad \|\Gamma_k\|_2 \leq 1, \text{ se } \lambda_k \geq 1/2,
	\end{equation}
	We split the analysis in these two cases. 
	\begin{itemize}
		\item[(a)] If $0 < \lambda_k < 1/2$, then 
		\begin{equation*}
		\max \{ \|\Gamma_k \|_2, 1 \} \leq \dfrac{1}{ 2 \sqrt{ \lambda_k - \lambda_k^2 }}.
		\end{equation*}
		From Assumption~\ref{Hip_lambda}, $\lambda_k = \|F_k\|_2^2$ and thus 
		\begin{align*}
		\max \{ \|\Gamma_k \|_2, 1 \} \|F_k\|_2 &\leq \dfrac{1}{ 2 \sqrt{ \lambda_k - \lambda_k^2 }} \|F_k\|_2 \\
		&= \dfrac{1}{ 2 \sqrt{ \|F_k\|_2^2 (1 - \|F_k\|_2^2) }} \|F_k\|_2 \\
		&= \dfrac{1}{ 2 \sqrt{ 1 - \|F_k\|_2^2 }} = \dfrac{1}{ 2 \sqrt{ 1 - \lambda_k }} \leq \frac{\sqrt{2}}{2}. %\quad \text{se} \quad 0 < \lambda_k < \frac{1}{2}.
		\end{align*}
		\item[(b)] If $\lambda_k \geq 1/2$, then $\max \{ \|\Gamma_k \|_2, 1 \} \leq 1$. 
		From the convergence of $x_k$ to $x_{\infty}$ for $k \in {\cal K}$ and continuity of $F$, it follows that 		
%		Por outro lado, considere $\varepsilon_0 >0$ tal que $y_k \in B(y_\infty, \varepsilon_0)$, para todo $k \geq k_0$. Veja que $\varepsilon_0$ sempre existe devido à convergência de $y_k$ para $y_\infty$, com $k \in \mathcal{K }$. Sendo $F$ contínua, do teorema de Weierstrass, é limitada no compacto $B(y_\infty, \varepsilon_0)$, isto é, podemos afirmar que $\|F_k\|_2 \leq M_F$, para todo $k\geq k_0$, para alguma constante $M_F >0$. 
	there exists a constant $M_F > 0$, such that $\|F_k\|_2 \leq M_F$, for all $k \in {\cal K}$. Hence 
		\begin{equation*}
		\max \{ \|\Gamma_k \|_2, 1 \} \|F_k\|_2 \leq \|F_k\|_2 \leq M_F, \quad \forall k \in {\cal K}.
		\end{equation*}
	\end{itemize}
	Therefore, from (a) and (b), we conclude that 
	\begin{equation*}
	\max \{ \|\Gamma_k \|_2, 1 \} \|F_k\|_2 \leq \max \left\lbrace \frac{\sqrt{2}}{2}, M_F \right\rbrace, \quad \forall k \in {\cal K}.
	\end{equation*}
and 
	\begin{equation}\label{norm_d_k1}
	\|d_k\|_2 \leq \frac{1}{\sqrt{\gamma}} \max \left\lbrace \frac{\sqrt{2}}{2}, M_F \right\rbrace =: M_1, \quad \forall k \in {\cal K}.
	\end{equation}
	%o teorema de Weierstrass (função contínua sobre um compacto atinge máximo e mínimo) podemos afirmar que $\|F_k\|_2 \leq M_F$, para todo $k\geq k_0$, para alguma constante $M_F >0$. Assim, conseguimos que
	%\begin{equation}\label{norm_d_k1}
	%\|d_k\|_2 \leq c_X \max \{c_\Gamma, 1  \} M_F =: M_1, \quad \forall k \geq k_0.
	%\end{equation}
%	Agora, se $k < k_0$ (número finito de índices), é claro que
%	\begin{equation}\label{norm_d_k2}
%	\|d_k\|_2 \leq \max_{k<k_0} \|d_k\|_2 =: M_2.
%	\end{equation}
%	Combinando (\ref{norm_d_k1}) e (\ref{norm_d_k2}) e denotando $\hat{M} = \max \{M_1, M_2\}$, segue que $\|d_k\|_2 \leq \hat{M}$, para todo $k \in \mathcal{K}$, como desejado.
This concludes the proof. 
\end{proof}

Now that it is proved that Algorithm~\ref{alg:LMM_armijo} generates a sequence of directions $\{d_k\}$ which is gradient-related, 
with the aid of \cite[Proposition~1.2.1]{Bertsekas1999}, the global convergence can be established 
(the reasoning in our proof is based on \cite[Theorem~3.1]{Dan2002}).

\begin{theorem}\label{prop_conv}
	Let $\{ x_k \}$ be a sequence generated by Algorithm~\ref{alg:LMM_armijo}. 
	Then, every limit point $\hat{x}$ of $\{ x_k \}$ is such that $\nabla \phi (\hat{x}) = \0$.
\end{theorem}
\begin{proof}
Let $K_1 = \{ k \in \mathbb{N} \mid \|F(x_k + d_k) \| \leq \vartheta \| F(x_k) \| \}$. 
If $K_1$ is infinite, it follows that $\| F(x_k) \| \rightarrow 0$, and thus any limit point $\hat{x}$ of $\{x_k\}$ is such that $F(\hat{x})=0$, hence $\nabla \phi(\hat{x})=\0$.
Otherwise, if $K_1$ is finite, let us assume, without loss of generality, that $\|F(x_k + d_k) \| > \vartheta \| F(x_k) \|$, for every $k$, 
such that the step-size is chosen to satisfy Armijo condition. Since, by Proposition~\ref{prop_dk}, the directions of Algorithm~\ref{alg:LMM_armijo} 
are gradient-related, it follows from \cite[Proposition~1.2.1]{Bertsekas1999} that any limit point $\hat{x}$ of $\{x_k\}$ is a stationary point of $\phi$. 
\end{proof}

The next result, which is similar to \cite[Theorem~3.1]{Yamashita} shows that if a limit point $\hat{x}$ of Theorem~\ref{prop_conv} is such that $F(\hat{x})=\0$, 
then $\alpha_k=1$ for every sufficiently large $k \in {\cal K}$ and $\{ \dist(x_k,X^*) \}$ converges to zero quadratically.

\begin{theorem}%[{{Yamashita e Fukushima \cite{Yamashita}}}]
	Assume that Assumptions~\ref{Hip_nullJL}-\ref{Hip_lambda} hold. 
	Let $\{ x_k \}$ be generated by Algorithm~\ref{alg:LMM_armijo}. 
	If a limit point $x^*$ of $\{x_k \}$ is such that $F(x^*)=0$, then the sequence $\{ \dist(x_k,X^*) \}$ converges to 0 quadratically, 
	given Assumption~\ref{Hip_J} and Assumption~\ref{Hip_error_bound} for this particular $x^*$ hold. 
\end{theorem}

\begin{proof}
	Since the limit point $x^*$ of $\{ x_k\}$ is such that $F(x^*)=0$, then there exists  $k_0 \in \mathbb{N}$ such that 
	\begin{equation}\label{theorem_2_aux}
	\|F(x_{k_0})\|_2 \leq \frac{c_2^2 \vartheta}{c_5 c_F}
	\end{equation}
	and 
	\begin{equation*}
	\|x_{k_0} - x^*\|_2 \leq r,
	\end{equation*}
	where %$\vartheta \in (0,1)$ é definida no Algoritmo \ref{alg:LMM_armijo} e 
	$r$ is the constant specified in Lemma~\ref{lemma_3}. 
Since $\{ \| F(x_k) \| \}$ is non-increasing, and from Lemma~\ref{lemma_3}, 
for every $k \geq k_0$, we have 
	\begin{equation}\label{theorem_2_auxB}
	\|F(x_{k})\|_2 \leq \frac{c_2^2 \vartheta}{c_5 c_F} \quad \text{ and } \quad x_k \in B(x^*, \delta).
	\end{equation}
%%From Lemmas~\ref{lemma_2} and \ref{lemma_3}, and Assumption~\ref{Hip_error_bound}, we have 
%%	\begin{equation*}
%%	\dist (y_{j+1},X^*) \leq c_5 \dist (y_{j},X^*)^2 \leq c_5 c_2^2 \|F (y_j)\|_2^2,
%%	\end{equation*}
%%	para todo $j$. 
%%	Seja $\bar{y}_{j+1}$ um ponto em $X^*$ tal que $\|y_{j+1} -\bar{y}_{j+1} \|_2 = \dist (y_{j+1},X^*)$. Segue então da desigualdade acima, de (\ref{F_approx}) e da Hipótese \ref{Hip_error_bound} que
%%	\begin{align}\notag
%%	\|F (y_{j+1})\|_2 &= \|F (y_{j+1}) - F(\bar{y}_{j+1})\|_2 \\ \notag
%%	&\leq c_F \|y_{j+1} -\bar{y}_{j+1} \|_2\\ \notag
%%	&= c_F \dist (y_{j+1},X^*)\\ %\notag
%%	&\leq c_5 c_F \dist (y_{j},X^*)^2\\ 
%%	&\leq \frac{c_5 c_F \|F (y_j)\|_2}{c_2^2} \|F (y_j)\|_2. \label{theorem_2_aux2}
%%	\end{align}

	For $k \geq k_0$, let $\bar{x}_{k+1} \in X^*$ be such that $\|x_{k+1} -\bar{x}_{k+1} \|_2 = \dist (x_{k+1},X^*)$. 
	Then, from \eqref{F_approx}, Assumption~\ref{Hip_error_bound}, Lemma~\ref{lemma_2} and \eqref{theorem_2_auxB}, it follows that 
	\begin{align}\notag
	\|F (x_{k+1})\|_2 &= \|F (x_{k+1}) - F(\bar{x}_{k+1})\|_2 \leq c_F \|x_{k+1} -\bar{x}_{k+1} \|_2\\ \notag
	&= c_F \dist (x_{k+1},X^*) \leq c_5 c_F \dist (x_{k},X^*)^2\\ 
	&\leq \frac{c_5 c_F \|F (x_k)\|_2}{c_2^2} \|F (x_k)\|_2 \leq \vartheta \|F (x_k)\|_2, \label{theorem_2_aux2}
	\end{align}
	which proves that $\alpha_k = 1$ for every $k \geq k_0$. 
	Hence, the quadratic convergence of $\{ \dist (x_k, X^*) \}_{k \geq k_0}$ to zero follows from Theorem~\ref{theorem_1}.		
%	Agora, seja $\{ y_k \}$ a sequência gerada por LMMSS com passo unitário e $y_0 := x_{k_0}$. Observe que $\{ y_k \}$ vem de (\ref{xk_passo_unitario}) e sua construção busca usufruir das propriedades obtidas para a convergência local. De fato, pelo Teorema \ref{theorem_1}, $\dist (y_j, X^*)$ converge para 0 quadraticamente. Portanto, é suficiente mostrar que $x_{k_0+j} = y_j$ para todo $j$, i.e., $\{ y_j \}$ satisfaz
%	\begin{equation*}
%	\|F (y_{j+1} )\|_2 \leq \vartheta \|F (y_j)\|_2,
%	\end{equation*}
%	para todo $j$. Isto implicaria, essencialmente, que a partir do índice $k_0$ somente passos de tamanho unitário são tomados, de modo que o teorema de convergência local pode ser aplicado. 
\end{proof}

\section{Application to parameter identification in heat conduction problems}\label{sec:app}

	In this section, we illustrate the effectiveness  of LMMSS in recovering 
	physical parameters  in heat conduction problems based on  temperature measurements.
	Two inverse  problems are considered, namely, the problem of  reconstructing a so-called
	perfusion coefficient in a 2D Pennes' bioheat model (Section~\ref{sec:perf}), and the problem of  reconstructing the thermal conductivity in a 2D heat conduction model (Section~\ref{sec:thermal}). 
	%These will be addressed separately as follows.
	
	We remark that due to noisy measurements the corresponding nonlinear least-squares problems may not have zero residue ($X^* = \emptyset$). 
	In this case, the local convergence theory of Section~\ref{sec:local} does not apply. 
	However, according to Section~\ref{sec:global}, the global convergence of Algorithm~\ref{alg:LMM_armijo} still holds, meaning that limit points are stationary for the sum-of-squares function. 

	Furthermore, in the presence of noise, a minimizer of the residue does not always correspond to the expect solution of the inverse problem and thus it is important to know when to stop the iterations of LMMSS (or LMM) before the quality of the iterates deteriorates.  
	In the following subsections we explain how to handle this issue by using a discrepancy principle  \cite{Morozov1984}  
	and show that LMMSS with an appropriate choice of $L$ leads to better solutions than the classic choice $L = I_n$ used in LMM.

	\subsection{Reconstruction of 2D perfusion coefficient}\label{sec:perf}

	We consider the application of LMMSS to the estimation of a blood perfusion coefficient 
	based on the combination of clinical temperature measurements with a 
	mathematical model for
	heat transport proposed by Pennes~\cite{pennes} and referred to as bioheat model. The 
	2D bioheat model in dimensionless form  involves a  
	partial differential equation (PDE)
	
	\begin{equation}
		U_t-\Delta U+P_fU=G,\quad 0<x<1, \; 0<y<{\rm L},\; t>0\label{eq1},
	\end{equation}
	supplemented by boundary and initial conditions
	\begin{eqnarray}
		&&U_{x}=0, \quad x=0,\; 0<y<{\rm L},\;t>0,\label{eq2}\\
		&&U_{x}=0,  \quad x=1,\; 0<y<{\rm L},\;t>0,\label{eq3}\\
		&&U_y=B(U-U_{\infty}),\quad y=0,\; 0<x<1,\; t>0, \label{eq4}\\
		&&U=0, \quad y={\rm L},\; 0<x<1,\; t>0,\label{eq5}\\
		&& U=U_0,\quad 0<x<1,\; 0<y<{\rm L},\; t=0\label{eq6}.
	\end{eqnarray}

	\noindent 
	In this model, $U(x,y,t)$ is the local tissue temperature,  $P_f(x,y)$ is the so-called 
	blood perfusion coefficient,  
	$G=G(x,y,t)$ is a source term which stands for metabolic and spatial heat generation, $B$ is the Biot
	number, $U_\infty$ is the environmental temperature  and $U_0=U_0(x,y)$
	denotes the initial temperature.  
	The boundary conditions (\ref{eq2})-(\ref{eq5})
	include  prescribed temperature in
	the upper skin surface (the boundary $y={\rm L}$), adiabatic conditions
	(in the boundaries $x=0$ and $x=1$) and convective heat transfer
	between the tissue and an adjoint large blood vessel (in the
	boundary $y=0$), as displayed in Fig.~\ref{region}.
	\begin{figure}
		\caption{Domain for perfusion estimation.}
		\centerline{
			\scalebox{.35}{\includegraphics{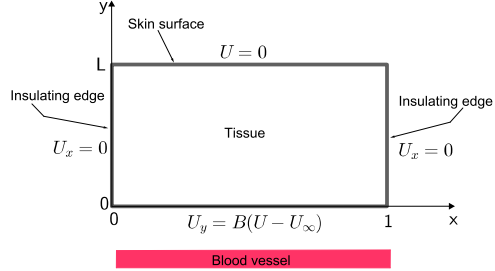}}
		}
		\label{region}
	\end{figure}
	
	The  problem under consideration consists of estimating the perfusion coefficient
	$P_f$  using the bioheat model and input data in the form of temperature measurements.
	In the approach used here, the bioheat model is transformed into a time-dependent semi-discrete system of ordinary differential equations
	involving perfusion coefficient values as a set of unknown 
	parameters, and the parameter estimation problem  is solved through LMMSS
	together with some regularization technique to stabilize the iterates.  For this we consider $(n+1)^2$ mesh points of the form $(x_i,y_j)$, the Chebyshev mesh, based on  Gauss-Lobatto points
	\begin{equation*}
		x_i = \frac{1}{2}\left[1-\cos \left( \frac{i\pi}{n} \right)\right] \quad \text{e} \quad y_j = \frac{1}{2}\left[1-\cos \left( \frac{j\pi}{n} \right)\right], \quad i,j=0,\dots,n,
	\end{equation*}
	and we use the Chebyshev pseudospectral method (CPM) to discretize  spatial derivatives. This yields the semi-discrete model,
	%More precisely, $U(\ell,t_k,\textbf{{p}})$ denote numerical solutions of a semi-discrete model %of the form
	\begin{equation}\label{ivpn}
		\left \{ \begin{array}{l}
			{\sf U}'(t)= {\sf A({\bf p})}{\sf U}(t) +  {\sf S}(t),\;\; t>0\\
			{\sf U}(0) =  U_0, \end{array} \right .  \end{equation}
	where ${\bf p}$ is a vector that contains values of the perfusion coefficient along the mesh, 
	%	with numerical solutions denoted by $U(t_k,\textbf{{p}})$, 
	${\sf A}({\bf p})$ is a square matrix
	that depends on ${\bf p}$ and  ${\sf S}(t)$ is a  vector valued function that incorporates information about the source term $G$ and boundary conditions.
	The reconstruction of the perfusion coefficient will be
	made pointwise by solving a nonlinear least squares problem with objective 
	function 
	\begin{equation*}
		\phi(\textbf{p}) =
		\frac{1}{2} \|{\sf U}(\textbf{p}) - \widetilde{\sf U} \|_2^2 \doteq
		\dfrac{1}{2} \sum_{k=1}^N \sum_{\ell=1}^M \| U(\ell,t_k,\textbf{{p}}) - \widetilde{U}(\ell,t_k) \|^2,
	\end{equation*}
	where $U(\ell,t_k,\textbf{{p}})$ denotes computed temperature values as functions 
	of parameter $\textbf{p}$ at prescribed locations inside the domain  at time $t_k$, and $\widetilde{U}(\ell,t_k)$ are measured temperature data at the same locations. As usually seen in inverse heat conduction problems,  we will assume that available data and exact data satisfy 
	$$\widetilde {\sf U} = {\sf U} + {\bf e},$$
	where ${\sf U}$ contains temperature values calculated by a second order predictor-corrector method that combines a Runge-Kutta method as predictor and Crank-Nicolson method as corrector, and 
	${\sf e}$ is a vector representing  inaccuracies. % and generated by Matlab 	in our numerical experiments.   
	For a complete analysis on the discretization procedure of the bioheat model we refer the reader to Bazán {\rm et. al}~\cite{Bazan17}.
	
	Jacobian matrices $J_k$ required along the minimization process are computed by solving  the so-called \textit{sensitivity problem}, where each column of $J_k$ comes from  solving an initial value problem (IVP) very similar to (\ref{ivpn}) involving  highly sparse matrices~\cite{Bazan17}. Moreover, since the nonhomogeneous terms for all IVPs are linearly independent, see \cite[Section 3.1]{BBazanLuchesi}, it can be proved that the Jacobian is full column rank. Thus, irrespective of the chosen 	scaling matrix $L$, 	the assumption ${\cal N}(J_k)\cap {\cal N}(L) = \{\bf 0\}$ is satisfied and hence the LMMSS iterates are well defined.
	
	As for  the scaling matrix, we choose 
	%	In this case LMMSS is implemented using  three types of {scaling} matrices. In fact, we take
	$L = {\cal L}_i$,  with
	
	\begin{equation}\label{scaling_L_i}
		\mathcal{L}_i = \left[ \begin{array}{c}
			I_{n}\otimes L_i(n+1)\\ 
			L_i(n) \otimes I_{n+1}
		\end{array}  \right],  \quad i = 1,2,3,
	\end{equation}
	where $\otimes$ represents the Kronecker product (or tensor product), 
	\begin{align*}
		\small
		L_1(n) &= \left[ \begin{array}{cccc}
			-1& 1 &  &  \\ 
			& \ddots & \ddots &  \\ 
			&  & -1 & 1
		\end{array}  \right] \in \R^{(n-1)\times n},  \\
		L_2(n) &= \left[ \begin{array}{ccccc}
			1& -2 & 1 &  &\\ 
			& \ddots & \ddots &\ddots & \\ 
			& & 1 & -2 & 1
		\end{array}  \right] \in \R^{(n-2)\times n},\\
		L_3(n) & = \left [\begin{array}{ccccccc} 
			-1 & 3& -3 & 1 &  & &\\
			& -1 & 3& -3 & 1 &  & \\
			& & \ddots & \ddots & \ddots &\ddots &\\
			& & & -1 & 3& -3 & 1\\
		\end{array}\right]\in \R^{(n-3)\times n},
	\end{align*}
	with $L_1$, $L_2$ and $L_3$ representing  discrete versions of the first, second  and third order derivative operators, respectively, widely used, for example, in image reconstruction problems and other applications~\cite{BazaBor,Bazan17,Hansen2006}. {By the way, ${\cal L}_i$ is singular as %it is well known that 
		\begin{align*}
			{\cal N}(L_1(n)) &=\text{span}\{[1,1,\dots,1]^T\},\\
			{\cal N}(L_2(n)) &= \text{span}\{[1,1,\dots,1]^T, [1,2,\dots,n]^T\},\\
			{\cal N}(L_3(n)) & =\text{span}\{[1,1,\dots,1]^T, [1,2,\dots,n]^T, [1^2,2^2,\dots,n^2]^T\}.
		\end{align*}
		The purpose of considering these 2D derivative operators
		into the numerical treatment of  this inverse problem is to illustrate  the impact of them on the numerical reconstruction when the exact
		solution is known to be smooth. In order to proceed, observe  that the estimation problem involves $(n+1)\times  n$ unknowns.
		%So  to keep the number of unknowns within reasonable
		%bounds we choose 
		Here we consider $n = 14$ (which means we deal with an optimization problem involving 210 unknowns), 
		$M=63$ measurements at internal locations of the domain and $N=8$ time levels. 
		Doing so, the Jacobian $J_k$ is of order $504\times 210$; for details about the choice of measurement temperature locations the reader is referred to Baz\'an {\rm et. al}~\cite{Bazan17}.

		In what follows, to illustrate the  impact that the regularizer choice causes on 
		the quality of the reconstructions from noisy data, we will conduct several numerical experiments intended to estimate the perfusion coefficient of a bioheat model with exact solution given by
		\begin{equation}\label{solpf2d}
			U(x,y,t)= \frac{ e^{ -\pi^2t}}{2 (B+{ \rm L}) }[((B+{ \rm L})y^2-By-{ \rm L}) \cos(\pi x)]+\frac{ BU_{\infty}}{ B+{\rm L}}({ \rm L}-y),
		\end{equation}
		for 
		$ B = 0.015, \; U_{\infty} = 0.001,\; {\rm L}=  1,$ and exact {\it smooth}  perfusion coefficient
		\begin{equation}\label{pf}
			P_f(x,y) =\sin(\pi xy).
		\end{equation}

		To this end we use temperature measurements with the noise vector {\bf e} %$\widetilde{\sf U}- {\sf U} = {\bf e}$ 
		containing zero mean Gaussian random numbers scaled  so that %, we will conduct numerical experiments with data so that 
		$\| {\bf e} \|/\| {\sf U}\| = \| {\sf U}-\widetilde{\sf U}\|/\| {\sf U}\|={\rm NL}$, where {\rm NL} 
		denotes the noise level in the data, and conduct the reconstruction process for several values of {\rm NL}. 
		
		It is worth mentioning that when solving inverse problems through iterative methods, 
		it is of fundamental importance to stop the iterations before the noise in the data starts to deteriorate their quality. 
		We achieved this objective through the discrepancy principle \cite{Morozov1984}: LMMSS iterations stop at the first index $j$ such that
		\begin{equation}
			\label{dp} \|{\sf U}({\bf p}^{(j)})-\widetilde{\sf U}\|\le \tau \|{\bf e}\|,
		\end{equation}
		where $\tau\ge 1$ is a safeguard parameter. 
		For our numerical experiments we take 	$\tau=1.05$, a slightly tight parameter as measure of good fit between the continuous problem and the discretization.  
		
		In order to assess the accuracy of recovered quantities, we use the relative error 	$${\rm RE}({\bf p}^{(j)}) = \|{\bf p}^{(j)}-{\bf p}\|/\|{\bf p}\|$$
		along interior points of the mesh, and the temperature reconstruction error defined by
		\begin{equation*}
			{\rm TRE} = \|{\sf U}({\bf p}^{(j)})-{\sf U}\|/\|{\sf U}\|.
		\end{equation*}
		
		We  report in Table~\ref{table:perfusion} numerical results  obtained with LMMSS using the singular scaling matrices  ${\cal L}_i$ from \eqref{scaling_L_i}, including the case $L = I$, for ${\rm NL} = 0.0001$ and  ${\rm NL} = 0.001$. 
		 
	Before commenting on them, let us highlight Figure~\ref{fig:relerr} where we monitor the relative error and absolute temperature error (which is proportional to the square root of the objective function) along the iterations of LMMSS with two different choices of $L$. 
	While the plot on the right  shows a monotone decay of objective function (residual norms), 
	which is in accordance with the theoretical properties of LMMSS, 
	on the left we see that, for $L={\cal L}_2$ (green curve), the relative error in the iterates ${\bf p}^{(j)}$ tends to deteriorate after the discrepancy principle is achieved (the threshold is represented by the red horizontal line) 	whereas for $L=I_n$ (blue curve) it seems to stabilize, but in a higher level. 
	Concerning the residual norm (right) we observe that stopping criterion \eqref{dp} is reached earlier for $L = {\cal L}_2$. 
		
		\begin{figure}[h]
			\centering
			\caption{Relative error ${\rm RE}({\bf p}^{(j)})$ (left)
				and  absolute temperature error (right) as functions of the iteration number, from data with  ${\rm NL} = 0.001$. 
				The blue dashed line corresponds to $L = I_n$ and the green dotted line to $L={\cal L}_2$. 
				The horizontal red line on the right represents the value $\tau\|{\bf e}\|=0.0039.$}
			\includegraphics[scale=0.92]{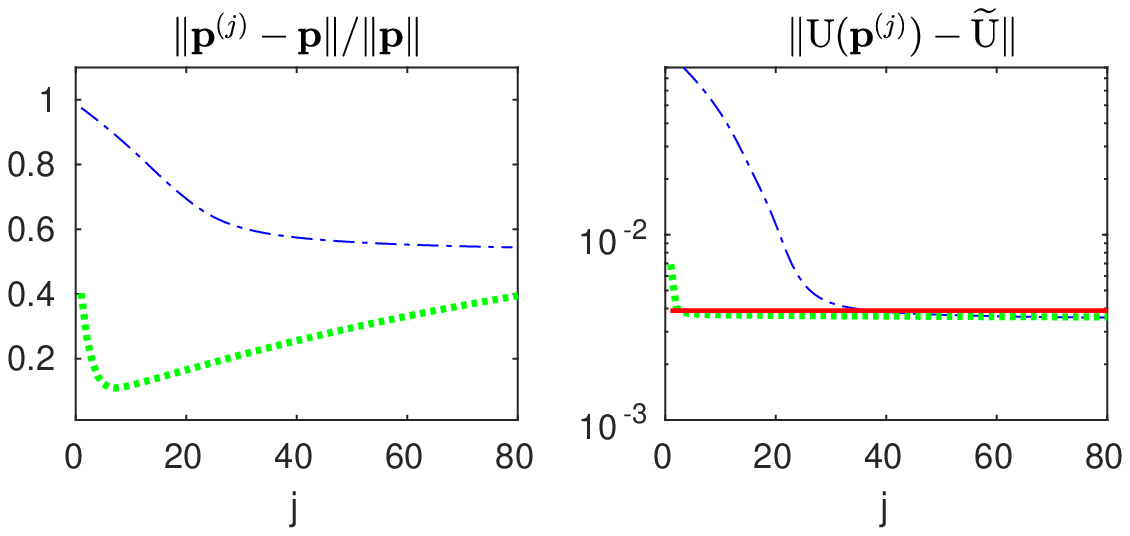} 
			\label{fig:relerr}
		\end{figure}
		
		As for the quality of the reconstruction of the perfusion coefficient measured by the relative errors in Table~\ref{table:perfusion}, we note that it ranges from acceptable to excellent when $L={\cal L}_i$, $i=1,2,3$, which is  also illustrated in Figure~\ref{fig:pf2d} where we present surface plots of the reconstructions.
		Notice that this fact should not be seen as a surprise, since the sought solution is smooth. Also  note that the results obtained with classic LMM (Algorithm~\ref{alg:LMM_armijo} using $L^T L = I_n$) are inferior both in terms of the quality of the reconstruction	 and the iterations count when compared to the other cases.	Other than that, what really surprises is the superior quality of the reconstructed temperature compared to the quality of the reconstructed perfusion coefficient.
		Indeed, the  occurrence of 
		small temperature reconstruction errors (TREs) simply highlight the ill-posed nature of the problem. That is,  the estimation problem
		allows for small temperature reconstruction errors in all cases, even with estimates of the perfusion coefficient differing greatly from the exact one.

		\begin{table}[h]
			\centering
			\caption{Results for perfusion coefficient recovery for two noise levels and $n+1=15$. }
			{\small 
				\begin{tabular}{lllllc}
					\toprule
					NL & {\bf Method} & $L$ & ${\rm RE}({\bf p}^{(j)})$  & TRE & \textbf{Iterations} \\ \midrule
					\multirow{4}{*}{0.001} & \textbf{LMM} & $I$  & 0.5860   & 4.0578e-04  & 36 \\
					& \textbf{LMMSS} & ${\cal L}_1$  & 0.3437  & 3.0196e-04  &  6 \\
					& \textbf{LMMSS} & ${\cal L}_2$  & 0.1718 &   2.3084e-04  &  3 \\
					& \textbf{LMMSS} & ${\cal L}_3$  & 0.1403 &   2.3898e-04  &  2 \\
					\midrule
					\multirow{4}{*}{0.0001} & \textbf{LMM} & $I $  & 0.5333 & 5.4241e-04  & 44 \\
					& \textbf{LMMSS} & ${\cal L}_1$  & 0.1539   &  3.7233e-05  &  7 \\
					& \textbf{LMMSS} & ${\cal L}_2$  & 0.0990   &  3.1243e-05  &  4 \\
					& \textbf{LMMSS} & ${\cal L}_3$  & 0.0516    &  2.6821e-05  &  3 \\
					\bottomrule
			\end{tabular}}
			\label{table:perfusion}
		\end{table}
		
		\begin{figure}[h]
			\centering
			\caption{Exact and recovered 2D perfusion coefficient. Noisy data correspond to  $\textup{NL} = 0.001$ (4 regularizers) and $\textup{NL} = 0.0001$ only for 
				${\cal L}_3$ (last plot from left to right). }
			\includegraphics[scale=0.7]{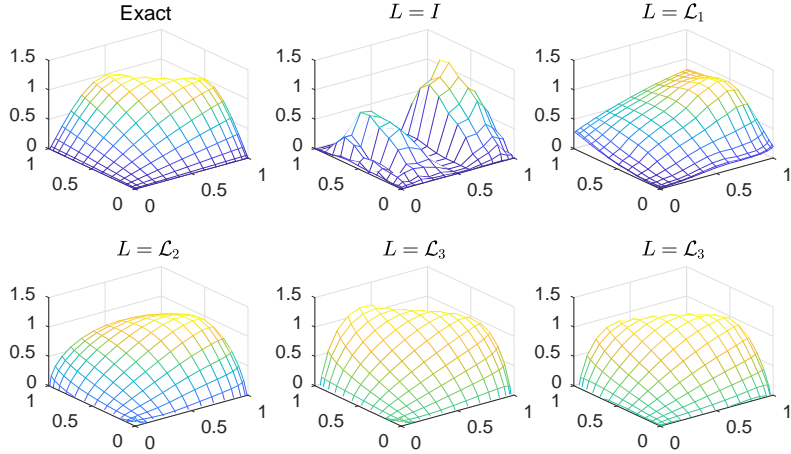} %\hspace{5pt}
			\label{fig:pf2d}
		\end{figure}

	}

	\subsection{Reconstruction of thermal conductivity}\label{sec:thermal}

	In this section, we illustrate the effectiveness of LMMSS in a thermal conductivity reconstruction problem aiming at recovering  the thermal conductivity incorporated as a parameter in a heat conduction model based on temperature measurements. For this we consider a two-dimensional
	heat conduction model in the finite domain  $\Omega \times [0,t_f]$, $t_f>0$, with $\Omega = (0,l_1)\times (0,l_2) \subseteq \R^2$, $l_1, l_2 >0$, described by 
	the partial differential equation (PDE) {\small
		\begin{equation}\label{eq:heat}
			C(x,y)\frac{\partial u}{\partial t} (x,y,t)=\nabla \cdot [{ K}(x,y)\nabla 
			u(x,y,t)] -q(x,y) u(x,y,t)+ g(x,y,t),
	\end{equation}}
	boundary conditions {\small 
		\begin{align}
			\label{eq:boundary1}
			-k_{11}(0,y)\frac{\partial u}{\partial x}(0,y,t) +h_1(y)(u(0,y,t)-f_1(y,t)) = 0, &\quad \text{on} \quad x = 0,
			\ y \in (0,l_2),\\
			\label{eq:boundary2}
			k_{11}(l_1,y)\frac{\partial u}{\partial x}(l_1,y,t) +h_2(y)(u(l_1,y,t) -f_2(y,t)) = 0,  &\quad \text{on} \quad
			x = l_1, \ y \in (0,l_2),\\
			\label{eq:boundary3}
			-k_{22}(x,0)\frac{\partial u}{\partial y}(x,0,t) +h_3(x)(u(x,0,t) -f_3(x,t)) = 0,   &\quad \text{on} \quad y = 0,
			\ x \in (0,l_1),\\
			\label{eq:boundary4}
			k_{22}(x,l_2)\frac{\partial u}{\partial y}(x,l_2,t) +h_4(x)(u(x,l_2,t)  -f_4(x,t)) = 0,   &\quad \text{on} \quad 
			y=l_2, \ x \in (0,l_1),
	\end{align}}
	%\end{document}
	for all $t\in (0,t_f]$, and initial condition 
	\begin{equation}\label{eq:Initial}
		u(x,y,0) = u_0 (x,y) \quad \text{on} \ \Omega,
	\end{equation}
	where $h_i$, $i = 1, \dots, 4$, are known as heat transfer functions, $f_i$, $i = 1,\dots,4$, heat flux functions and $u_0(x,y)$ is the initial temperature distribution. 
	The other variables stand for heat capacity $C(x,y)>0$, reaction term $q(x,y) \geq 0$, source term $g(x,y,t)$, temperature values $u(x,y,t)$ and thermal conductivity 
	\begin{equation}\label{eq:conductivity}
		{ K}(x,y) = \left[ \begin{array}{cc}
			k_{11} (x,y) & 0 \\ 
			0 & k_{22} (x,y)
		\end{array}  \right],
	\end{equation}
	{where $k_{11}$ and $k_{22}$ are positive continuous functions.}
	
	In the {\it direct problem} associated with the heat conduction model described above, given the thermal conductivity $K$ and the other physical parameters, 
	the goal is to determine the temperature $u(x,y,t)$.  As for the recovery problem, it can be stated as follows:
	given a set of temperature measurements inside the spatial domain, we want to recover estimates of the parameter $K$ considering the other parameters involved in the model as input data. For a wide range of applications of the conductivity recovery problem on industry and engineering the reader is referred to~\cite{Alessandrini17, AltintasB, Luchesi, Komanduri, Ozisik1993, Alifanov12}.
	
	Clearly, since $K$ depends nonlinearly on $u$, what we face is a nonlinear inverse problem. 
	Several parameter estimation methods which combine  numerical methods for PDEs such as the finite difference method (FDM), 
	the finite element method (FEM) or  the boundary element method (BEM), and optimization techniques 
	%as conjugate gradient method and the classical LMM,
	have been developed to address the conductivity estimation problem, e.g. 
	\cite{Mera00, Mahmood19, Pasdunkorale03, Alessandrini17, Cao2018}. In all cases
	the direct problem must be solved many times and therefore choosing a method to solve PDEs is crucial.
	The conductivity recovery method that we will illustrate in this section 
	follows this line of action. In fact, it combines an efficient method to solve the 
	direct problem (\ref{eq:heat})--(\ref{eq:Initial}) based on the
	Chebyshev pseudospectral method together with the Crank-Nicolson (CN) method as 
	time integrator, and LMMSS as optimization tool for estimating the desired parameter. 
	Our decision to choose such a PDE solver  is supported by the fact that the former is 
	is know to produce highly accurate
	approximations of spatial derivatives~\cite{BBazanLuchesi, B_Luchesi_Bazan}, while 
	the latter is second order 
	accurate and absolutely stable \cite{Crank_Nicolson_1947}.
	
	As in the previous application, we consider $(n+1)^2$ mesh points based on  Gauss-Lobatto points
	and use CPM to discretize  spatial derivatives on $\Omega$. This  transforms  model (\ref{eq:heat})--(\ref{eq:Initial}) into a 
	time-dependent system of ordinary differential equations (ODEs), the semi-discrete model,
	\begin{equation}\label{semi-disc}
		\left\{ \begin{aligned}
			& \textbf{C} \textbf{v}'(t) = \textbf{Av}(t) + \textbf{S}(t), \ t>0 \\ 
			&\textbf{v} (0) = \textbf{u}^0
		\end{aligned} \right.,
	\end{equation}
	where, roughly speaking, $\bf A$ contains spatial discretization terms 
	involving conductivity values $K(x_i,y_j)$, $\bf S$ contains terms coming from the source $g$ and boundary information, $\bf C$ contains discrete heat capacity values, and 
	$\textbf{u}^0$ consists of initial temperature values. Then, in order to  determine temperature values at several time steps we
	use the Crank-Nicolson method.

	As for the inverse problem, it is formulated
	as a nonlinear least squares problem involving temperature measurements as input data, usually contaminated with errors caused by imprecision in physical experiments or by 
	external noise. The functional to be minimized is described by
	\begin{equation*}
		\phi(\textbf{k}) =
		\frac{1}{2} \|{\sf u}(\textbf{k}) - \widetilde{\sf u} \|_2^2 \doteq
		\dfrac{1}{2} \sum_{k=1}^N \sum_{j=0}^n \sum_{i=0}^n \| u(x_i,y_j,t_k,\textbf{{k}}) - \widetilde{u}(x_i,y_j,t_k) \|^2,
	\end{equation*}
	{where %$\textbf{k}$ is a $2(n+1)^2\times 1$ vector containing  unknowns $K(x_i,y_j)$ related to $k_{11}$ and $k_{22}$, i.e.,
		\begin{equation*}
			\textbf{k} = \left[ \begin{array}{c}
				\textbf{k}_{11}	\\
				\textbf{k}_{22}
			\end{array} \right] \in \R^{2(n+1)^2}
		\end{equation*}
		contains unknowns $k_{11}(x_i,y_j)$ and $k_{22}(x_i,y_j)$ along the mesh points
		$(x_i,y_j)$,
		${\sf u}(\textbf{k})$ denotes an $[(n+1)^2N]\times 1$ block  vector containing  
		solutions of the direct problem for given ${\bf k}$,  with the $k$-th  block 
		being associated with  time level $t_k$, $k=1,\dots, N$, and
		$\widetilde{\sf u}$ is a vector of measured  temperature data at the $N$ time steps
		with the same structure as ${\sf u}(\textbf{k})$.}
	Jacobian matrices $J_k$ required along the minimization process are computed, as in the previous subsection, by solving  the sensitivity problem. For details the reader is referred to \cite[Sections 3.1 and 4]{B_Luchesi_Bazan} and \cite[Section 3.1]{BBazanLuchesi}.
	
	Here, we choose  {scaling} matrices %. In fact, we take 
	$L = \tilde{\cal L}_i$, 
	\begin{equation}\label{nscaling_L_i}
		\tilde{\cal L}_i = \left[ \begin{array}{c}
			I_{n+1}\otimes L_i(n+1)\\ 
			L_i(n+1) \otimes I_{n+1}
		\end{array}  \right],  \quad i = 1,2.
	\end{equation}
	which differs
	slightly from that used in the previous application (see \eqref{scaling_L_i}) because of the boundary conditions, but following the same ideas.  
	
	As in the previous example, we will assume that the received data $\widetilde{\sf u}$ and the exact data ${\sf u}$ satisfy
	\begin{equation*}
		\widetilde{\sf u}=  {\sf u}+{\sf e},%\;\;\text{with}\;\;  \norm{\widetilde{\sf u} - {\sf u}} = \textup{NL} \norm{\sf u},
	\end{equation*}
	where {${\sf e}$} contains zero mean Gaussian random numbers scaled  so that  $\norm{\widetilde{\sf u} - {\sf u}}/\norm{\sf u} = \textup{NL}$, 
	where {\rm NL} denotes the noise level. 
	%containing zero mean random numbers scaled  so that $\| {\sf U}-\widetilde{\sf U}\|/\| {\sf U}\|={\rm NL}$, where {\rm NL} denotes the noise level
	In the case of $\textup{NL} = 0$, we stop the iterations at ${\bf k}^{(j)}$ if the gradient of the objective function is small, i.e., 
	\begin{equation*}
		\norm{\nabla \phi ({\bf k}^{(j)})} < \varepsilon, %\quad j \geq 1,
	\end{equation*}
	or if the relative variation of the iterates is small, %that is, we stop when
	\begin{equation*}
		\frac{\norm{{\bf k}^{(j)} - {\bf k}^{(j-1)}}}{ \norm{{\bf k}^{(j)}}} < \varepsilon, %\quad j \geq 2,
	\end{equation*}
	both with $\varepsilon = 5\times 10^{-4}$ in all numerical examples. Otherwise, if $\textup{NL} \neq 0$, the iterations stop using the discrepancy principle \cite{Morozov1984} as in (\ref{dp}) %a regularization technique to control noise propagation: we stop at the first $j$ such that 
%	\begin{equation*}
%		\|{\sf u}({\bf k}^{(j)}) - \widetilde{\sf u} \|_2 \lessapprox \tau \|{\sf e}\|_2,
%	\end{equation*}
	for $\tau = 1.1$ in the numerical experiments of this section.
	
	In the next subsections we will discuss examples and  compare results obtained 
	with classic LMM and LMMSS using data with and without noise. The accuracy of recovered conductivities
	is measured by relative errors, %\vskip 0.5pc
	\begin{equation*}
		{\rm RE}({\bf k}_{11}^{(j)})=\|{\bf k}_{11}^{(j)} - {\bf k}_{11}\|_2/\|{\bf k}_{11}\|_2   \quad \text{and} \quad {\rm RE}({\bf k}_{22}^{(j)})=\|{\bf k}_{22}^{(j)} - 
		{\bf k}_{22}\|_2/\| {\bf k}_{22}\|_2,
	\end{equation*}
	as well as by temperature reconstruction errors \text{TRE} defined similarly as in the perfusion estimation problem. %by %\vskip 0.25pc %
	%\begin{equation*}
	%	\text{TRE} = {\|{\sf u}({\bf k}^{(j)}) - {\sf u}\|_2}/{\|{\sf u}\|_2}.
	%\end{equation*}
	%$\qquad \text{TRE} = {\|{\sf u}({\bf k}^{(j)}) - {\sf u}\|_2}/{\|{\sf u}\|_2}.$
	
	To rigorously evaluate the results obtained in the numerical simulations, two test problems will be solved 30 times using different noisy temperature data. We will report average relative errors of reconstructed quantities, including the maximum number of iterations of all instances, denoted by MI, spent until some convergence criterion is met.
	In both  examples we consider $N = 10$ equally spaced time stages in $[0,t_f]$ and 
	$n+1 = 16$ grid points in each direction, which brings 256 conductivity unknowns $K(x_i,y_j)$ on $\Omega $. 
	%From the implementation point of view, it is important to emphasize that most of the matrices involved are sparse, reducing numerical effort. 

	\subsubsection{Example 1: isotropic conductivity}
	
	%\begin{figure}[b]
	%	\centering
	%	\caption{Temperatura reconstruída via FEM em alguns estágios temporais.}
	%	\includegraphics[scale=0.75]{figures/CPM/T3temp.eps}
	%	\legend{Fonte -- Boos, Luchesi e Bazán \cite{B_Luchesi_Bazan}, 2020.}
	%	\label{fig:temp_FEM}
	%\end{figure}
	
	Extracted from Mahmood and Lesnic \cite{Mahmood19}, this example considers as sought solution $K(x,y)$ an isotropic conductivity, that is,
	\begin{equation*}
		k_{11} (x,y) = k_{22} (x,y) := k(x,y) = \dfrac{1+x+y}{12},
	\end{equation*}
	the problem being defined on $\Omega = (0,1) \times (0,1)$, with final time $t_f = 1$. The remaining parameters are:
	\begin{gather*}
		C(x,y) = 1,\quad g(x,y,t)=0,\quad f_1(y,t)= -1,\quad f_2(y,t)= 1,\\
		f_3(x,t) = -1, \quad f_4(x,t)=1, \quad \text{and} \quad h_i \equiv 0,\  i=1,\dots, 4,
	\end{gather*}
	and initial condition $u_0(x,y)=0$. For this test problem the exact solution
	$u(x,y,t)$ is not  available and  temperature data for testing LMMSS 
	should be generated artificially. 
	We make this by 
	%, which would also provide data for the inverse problem. 
	applying a finite element method (FEM) to the model (\ref{eq:heat})--(\ref{eq:Initial}),  using  a refined mesh so as to build  a highly accurate approximate solution $u$,  which is then   interpolated to the Chebyshev mesh 
	to generate the sought ``exact'' solution. For more details on this procedure, see \cite[Section 3.2]{B_Luchesi_Bazan}.

	Numerical results  obtained from data with {$\textup{NL}=0.001$ (i.e. data
		with relative noise level  0.1\%)},  scaling matrices 
	$L = I$ and $L = \tilde{\cal L}_i$ as in (\ref{nscaling_L_i}) are shown in
	Table \ref{table:plane} and
	Figure \ref{fig:plane_iso}. In all cases, for the  initial guess we choose  $k^0(x,y) = 1/4$. As it can be seen, both scaling matrices used in this experiment
	yield good quality reconstructions, with reconstruction errors that do 
	not exceed 10\% of relative error but with an advantage in favor of LMMSS in terms of accuracy and computational effort.

	\begin{table}[h!]
		%\vskip -1pc
		\centering
		\caption{Results of one instance for different scaling matrices.} 
		%Source: adapted from Boos, Luchesi and Bazán \cite{B_Luchesi_Bazan}.}
		{\small
			\begin{tabular}{llll}
				\toprule
				\textbf{Scaling matrix} & $I$  & $\tilde{\mathcal{L}}_1$ & $\tilde{\mathcal{L}}_2$\\
				%$\alpha$ & 0.1 & 0.35 & 0.1 & 0.1\\
				\midrule
				${\rm RE}({\bf k}^{(j)})$ & 0.0946  & 0.0128  &  0.0134  \\
				%IE & 0.0236  &  0.0096 &  0.0084  \\
				MI & 7 & 4 & 4  \\
				\bottomrule
		\end{tabular}}
		%\caption{Errors and number of iterations until DP is satisfied, for different choices of ${\bf D}$ and $\alpha$, considering as exact $k(x,y) = (1+x+y)/12$.}
		%	\legend{}
		\label{table:plane}
	\end{table}
	
	Also note the visual cohesion and smoothness displayed in in Figures \ref{fig:planed} and \ref{fig:planee} (effect of the derivative operator), contrasting with the presence of ``peaks''  in the reconstruction obtained with $L=I$ (classic LMM) displayed in Figure \ref{fig:planeb} .
	% such matrices seek to preserve, as previously mentioned, that the solution is built with smooth vectors from $L$, contributing to the visual cohesion present in Figures \ref{fig:planed} and \ref{fig:planee}. In a way, we can understand that the variables ${\bf k}^{(j)}$ are constructed respecting the smoothness relationship imposed by the scaling matrices and not individually as for classical LMM, which then generates the ``peaks'' present in Figure \ref{fig:planeb}. 
	
	\begin{figure}[ht!]
		\centering
		\caption{Comparison between iterates generated by LMM and LMMSS.}
		%Source: Boos, Luchesi and Bazán \cite{B_Luchesi_Bazan}.}
		%\hspace{-14pt}
		\subfloat[Exact]{\scalebox{.53}{\includegraphics{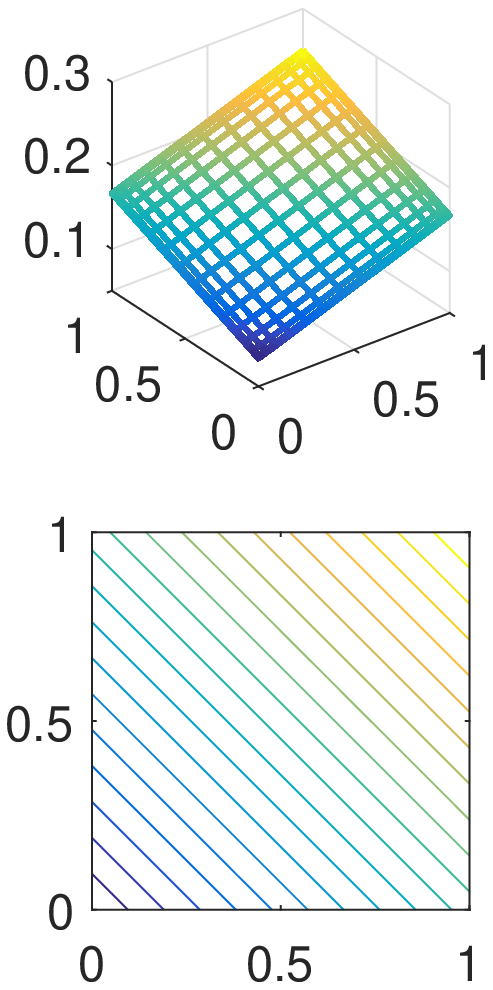}}} \quad
		\subfloat[$L = I$ \label{fig:planeb}]{\scalebox{.53}{\includegraphics{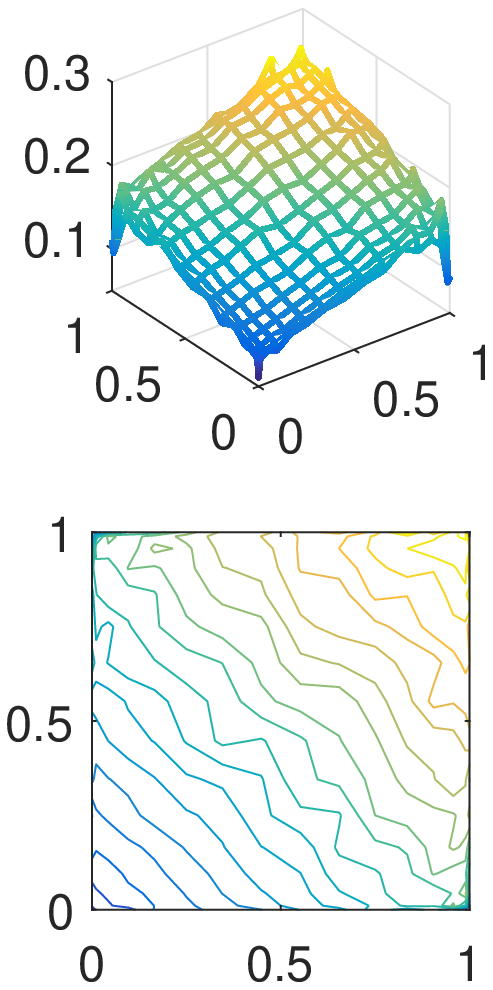}}} \quad
		%\subfloat[$\textbf{D} = I$, $\alpha = 0.35$ \label{fig:planec}]{\scalebox{.55}{\includegraphics{figures/CPM/iso_plane_Id_035.eps}}} \quad
		\subfloat[$L = \tilde{\mathcal{L}}_1$ \label{fig:planed} ]{\scalebox{.53}{\includegraphics{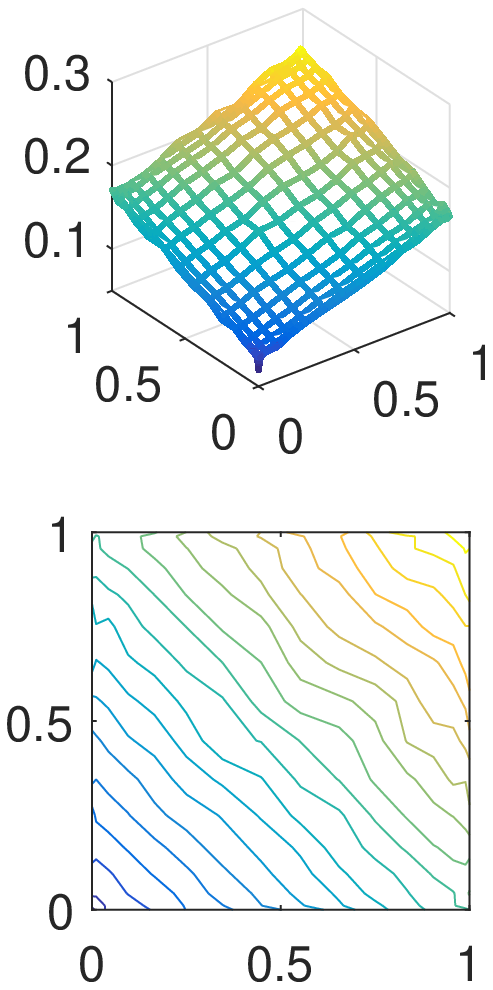}}} \quad
		\subfloat[$L = \tilde{\mathcal{L}}_2$ \label{fig:planee} ]{\scalebox{.53}{\includegraphics{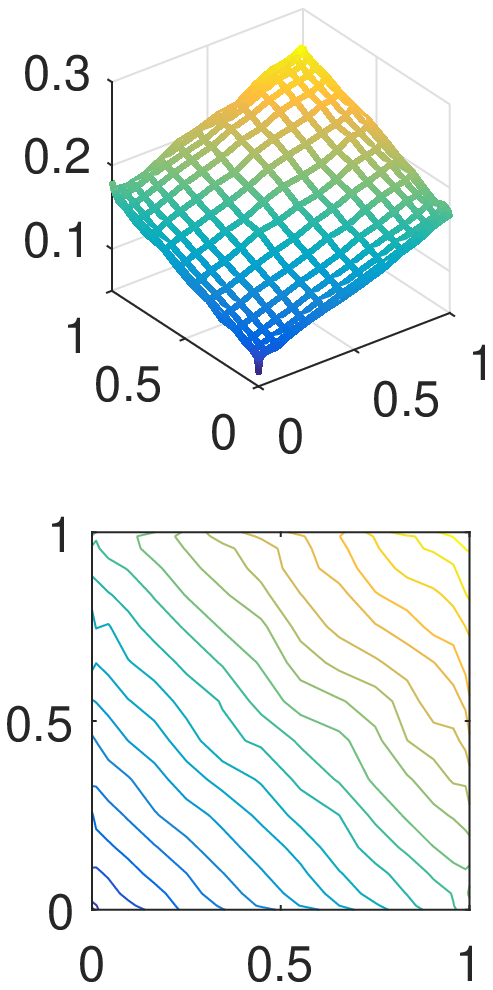}}}
		\label{fig:plane_iso}
		%	\legend{}
	\end{figure}

	\subsubsection{Example 2: orthotropic conductivity}
	
	This example is taken from  the work of Cao, Lesnic and Cola\c co \cite{Cao2018}, adapted to model (\ref{eq:heat})--(\ref{eq:Initial}), defined in $\Omega = (0,1) \times (0,1)$ and with final time $t_f = 1$. In this case the exact solution of
	the model is available and defined as 
	%has a complete description of the components of the direct problem, namely:
	\begin{equation*}
		u(x,y,t) = e^{-t} (\sin (\pi x) \sin (\pi y) + (\pi +1) (x+y) + 1), %\quad \text{em } (0,1) \times (0,1)\times [0, t_f],
	\end{equation*}
	and the other model parameters are given as 
	$h_i = 1, \ i = 1, \dots, 4$,  $C(x,y) = 1$,  $q(x,y) = 0$,
	\begin{align*}
		f_1(y,t) &= -\frac{1+y}{12} e^{-t}(\pi \sin (\pi y) + \pi +1) + e^{-t} ((\pi +1)y +1),\\
		f_2(y,t) &= \frac{2+y}{12} e^{-t}(-\pi \sin (\pi y) + \pi +1) \\
		&\quad + e^{-t} (-\sin (\pi y) + (\pi+1) (1+y) +1 ),\\
		f_3(x,t) &= -\frac{1+0.5x}{12} e^{-t}(\pi \sin (\pi x) + \pi +1) + e^{-t} ((\pi +1)x +1),\\
		f_4(x,t) &= \frac{2+0.5x}{12} e^{-t}(-\pi \sin (\pi x) + \pi +1) \\
		&\quad + e^{-t} (-\sin (\pi x) + (\pi+1) (1+x) +1 ),
	\end{align*}
	source term
	\begin{align*}
		g(x,y,t) &= -e^{-t} (\sin (\pi x) \sin (\pi y) + (\pi +1) (x+y) + 1)\\
		&\quad -\frac{e^{-t}}{12}[ 2\pi +2+\pi \sin (\pi(x+y)) ]\\
		&\quad +\frac{\pi^2 e^{-t}}{12} (2+1.5x+2y) \sin (\pi x)\sin (\pi y)
	\end{align*}
	and, orthotropic conductivities
	\begin{equation}\label{k_cao}
		k_{11} (x,y) = \frac{1+x+y}{12} \quad \text{and} \quad k_{22} (x,y) = \frac{1+0.5x+y}{12}.
	\end{equation} 
	
	For this test problem, the iterates start with 
	$k^0_{11}(x,y) = k^0_{22}(x,y) = 1/4$ and the reconstruction problem is addressed
	{for $\textup{NL} = 0$, $\textup{NL}= 0.001$ and $\textup{NL}=0.01$.} 
	As we are dealing with an orthotropic conductivity, the scaling matrices come  in the form
	\begin{equation*}
		L = I_2 \otimes \tilde{\mathcal{L}}_i, \quad i = 1,2,
	\end{equation*}
	to ensure the smoothing effect of $\tilde{\mathcal{L}}_i$  on both  $k_{11}$ and $k_{22}$ \cite{BBazanLuchesi}. 
	
	Average results are displayed in Table \ref{table:cao_1over4}. 
	From the table it is apparent that LMM does not produce good quality results with respect to relative errors to the exact solution, even in the case of noiseless data, which contrasts with the reconstructions obtained by the proposed method, with relative errors varying between 2\% and 11\% for $\textup{NL}=0.001$, and between 3\% and 20\% for $\textup{NL}=0.01$, which is very reasonable mainly for we deal with the numerical treatment of a nonlinear ill-posed problem from noisy data. Here it is worth noting that despite the quality of the reconstructions being very different, 	the same does not happen with the temperature reconstruction errors (TRE), which are small and close to each other for both methods.  This is nothing more than the evidence that we are dealing with an ill-posed problem. In the noiseless case this also indicates the sum-of-squares function has more than one global minimizer ($X^*$ is not a singleton).

	\begin{table}[h]
		\centering
		\caption{Average results found for three noise levels. }%Source: Boos, Bazán and Luchesi \cite{BBazanLuchesi}.} %All instances stopped %with DP.}
		{\small
			\begin{tabular}{lllllll}
				\toprule
				%			NL & {\bf Método} & ${\rm RE}({\bf k}_{11}^{(j)})$ & ${\rm RE}({\bf k}_{22}^{(j)})$ & TRE & MI \\ \midrule
				%			\multirow{2}{*}{0\%} & \textbf{LMM} ($L = I$)  &  0.2937  &  0.3698  &  0.0000  &  13 \\
				%			& \textbf{LMMSS} ($L = I_2 \otimes \mathcal{L}_2$)  &  0.0291  &  0.0127  &  0.0000  &   8 \\ \midrule
				%			\multirow{2}{*}{0.1\%} & \textbf{LMM} ($L = I$) &  0.3996  &  0.5211  &  0.0063  &   4 \\
				%			& \textbf{LMMSS} ($L = \mathcal{L}_2$)  &  0.0611  &  0.1138  &  0.0100  &   2 \\ \midrule
				%			\multirow{2}{*}{1\%} & \textbf{LMM} ($L = I$) &  0.5100  &  0.6851  &  0.0398  &   1 \\
				%			& \textbf{LMMSS} ($L = \mathcal{L}_2$) &   0.1446  &  0.2024  &  0.0237  &   1 \\
				NL & {\bf Method} & $L$ & ${\rm RE}({\bf k}_{11}^{(j)})$ & ${\rm RE}({\bf k}_{22}^{(j)})$ & TRE & MI \\ \midrule
				\multirow{3}{*}{0} & \textbf{LMM} & $I$  &  0.2937  &  0.3698  &  0.0000  &  13 \\
				& \textbf{LMMSS} & $I_2 \otimes \tilde{\mathcal{L}}_1$  &  0.0195  &  0.0154  &  0.0000  &  6 \\
				& \textbf{LMMSS} & $I_2 \otimes \tilde{\mathcal{L}}_2$  &  0.0291  &  0.0127  &  0.0000  &   8 \\ \midrule
				\multirow{3}{*}{0.001} & \textbf{LMM} & $I$ &  0.3996  &  0.5211  &  0.0063  &   4 \\
				& \textbf{LMMSS} & $I_2 \otimes \tilde{\mathcal{L}}_1$  &  0.0218  &  0.0185  &  0.0003  &  3 \\
				& \textbf{LMMSS} & $I_2 \otimes \tilde{\mathcal{L}}_2$  &  0.0611  &  0.1138  &  0.0100  &   2 \\ \midrule
				\multirow{3}{*}{0.01} & \textbf{LMM} & $I$ &  0.5100  &  0.6851  &  0.0398  &   1 \\
				& \textbf{LMMSS} & $I_2 \otimes \tilde{\mathcal{L}}_1$  &  0.0388  &  0.0318  &  0.0022  &  2 \\
				& \textbf{LMMSS} & $I_2 \otimes \tilde{\mathcal{L}}_2$ &   0.1446  &  0.2024  &  0.0237  &   1 \\
				\bottomrule
		\end{tabular}}
		%	\legend{}
		\label{table:cao_1over4}
	\end{table}
	
	Figure \ref{fig:cao_plot_line_IDL2} presents the average results obtained for $k_{11}$ for different $y$ values and $\textup{NL} = 0.001$. Once more it can be seen that while LMMSS produces solutions of good quality,  LMM brings
	solutions with  high oscillations and reconstructions far from the desired. Similar results were found for $k_{22}$. Here it is worth noting that among the scaling matrices tested in this simulation, $L = I_2 \otimes \tilde{\mathcal{L}}_1$  was the one that produced the best results. We believe that this is due to the fact that as the solution is linear, a second  order derivative operator would tend to diminish the importance of the term $\|L{\bf k}\|$ in the minimization process.

	\begin{figure}[t]
		\centering
		\caption{Average results for $k_{11} (x,y)$ for some fixed $y$ values and $\textup{NL} = 0.001$. In the legends, $L_1$ and $L_2$ represent, respectively, $I_2 \otimes \tilde{\mathcal{L}}_1$ and $I_2 \otimes \tilde{\mathcal{L}}_2$.}% Source: Boos, Bazán and Luchesi \cite{BBazanLuchesi}.}
		\includegraphics[scale=0.4]{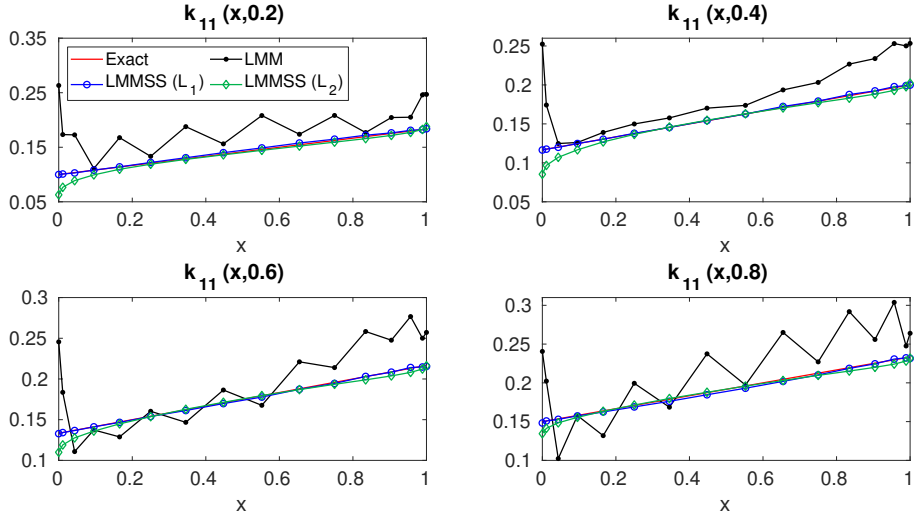}
		%	\legend{}
		\label{fig:cao_plot_line_IDL2}
	\end{figure}
	
	%In summary, we can see that the technique through LMMSS with scaling matrices emulating derivative operators presents a significant improvement in the obtained solutions when compared to LMM, reducing the amount of iterations necessary to produce advantageous approximations. In words, these $L$ choices give priority to writing the iterates containing smooth components, resulting in the cohesion present in the figures. Note that both this example and the one in the previous subsection have continuous and differentiable conductivities, information that LMM with $L = I$ does not carry to the iterative process, but that LMMSS takes advantage of by choosing $L$ appropriately. 

%\newpage
\section{Conclusion}\label{sec:final}
Motivated by the improved performance of the Tikhonov regularization method in general form when  solving discrete ill-posed problems with smooth solution~\cite{Hansen1998},
we proposed a version of the Levenberg-Marquardt method with singular scaling matrix
which we named LMMSS.
Under a  completeness condition often used in the analysis of linear ill-posed problems we showed that the LMMSS iterates are well defined and  establish that
it has local quadratic convergence under an error bound assumption for the zero residue case. 
We also prove that the search directions $d_k$ are gradient-related and with that
we ensure that  the limit points of the sequence generated by LMMSS are stationary points of the sum-of-squares function. 
%As a result, in a similar fashion as developed by Yamashita and Fukushima \cite{Yamashita}, we propose a globally convergent algorithm with quadratic rate in the vicinity of a minimizer. 

As potential applications, we study two parameter identification problems involving a 2D heat conduction models. 
% described by a partial differential equation with mixed boundary conditions and initial condition. The problem is discretized using the Chebyshev pseudospectral method for the spatial variables and Crank-Nicolson as time stepping method, as described in \cite{BBazanLuchesi,B_Luchesi_Bazan}. 
Since in these applications the parameter to be identified is supposed to be smooth, in LMMSS we introduce singular scaling matrices chosen as discrete derivative operators aiming to promote smoothness in the reconstructed parameter. % to incorporate in the iterative process smoothing properties of conductivities in the horizontal and vertical directions. 
To mitigate the effect of inaccuracies in the data, the discrepancy principle is used as stopping criterion. 
Numerical results using synthetic data illustrate that LMMSS can be useful in
practical applications, mainly due to the high quality of the reconstructions and the low operational cost. 
However, continued experience with LMMSS in other applications is necessary to fully  assess its potential.

For now, our convergence analysis only covers the case where $L$ is fixed through the iterations. 
This is in line with some applications, in which the proposal is to introduce desired properties present (or expected) in the exact solution. 
In such scenario, matrix $L$ does not need to change in each iteration. 
Nevertheless, it should be interesting in other applications to consider an iteration-dependent $L_k$. 
The convergence analysis in this case shall be subject of a future investigation.

Another future objective is to address the local convergence of LMM with singular scaling in the context of nonzero-residue problems by extending the results in \cite{Behling2019}.

Besides, in order to consider large-scale problems, strategies for exploiting sparsity and special structure for solving the LMM linear system as well as inexact solution of the LMM subproblem will also be considered in future works.

%\bmhead{Supplementary information}
%
%If your article has accompanying supplementary file/s please state so here. 
%
%Authors reporting data from electrophoretic gels and blots should supply the full unprocessed scans for key as part of their Supplementary information. This may be requested by the editorial team/s if it is missing.
%
%Please refer to Journal-level guidance for any specific requirements.

\section*{Acknowledgments}

%\tred{The work of EB was supported by CAPES and FAPESC, grant number 88887.178114/2018-00. 
%DSG thanks CNPq (Conselho Nacional de Desenvolvimento Científico e Tecnológico), grant 305213/2021-0.}
This work was partially supported by Brazilian agencies CAPES (Coordenação de Aperfeiçoamento de Pessoal de Nível Superior),  FAPESC (Fundação de Amparo à Pesquisa e Inovação do Estado de Santa Catarina), grant number 88887.178114/2018-00 and CNPq (Conselho Nacional de Desenvolvimento Científico e Tecnológico), grant 305213/2021-0.

%\section*{Declarations}
%
%Some journals require declarations to be submitted in a standardised format. Please check the Instructions for Authors of the journal to which you are submitting to see if you need to complete this section. If yes, your manuscript must contain the following sections under the heading `Declarations':
%
%\begin{itemize}
%\item Funding
%\item Conflict of interest/Competing interests (check journal-specific guidelines for which heading to use)
%\item Ethics approval 
%\item Consent to participate
%\item Consent for publication
%\item Availability of data and materials
%\item Code availability 
%\item Authors' contributions
%\end{itemize}
%
%\noindent
%If any of the sections are not relevant to your manuscript, please include the heading and write `Not applicable' for that section. 
%
%%%===================================================%%
%%% For presentation purpose, we have included        %%
%%% \bigskip command. please ignore this.             %%
%%%===================================================%%
%\bigskip
%\begin{flushleft}%
%Editorial Policies for:
%
%\bigskip\noindent
%Springer journals and proceedings: \url{https://www.springer.com/gp/editorial-policies}
%
%\bigskip\noindent
%Nature Portfolio journals: \url{https://www.nature.com/nature-research/editorial-policies}
%
%\bigskip\noindent
%\textit{Scientific Reports}: \url{https://www.nature.com/srep/journal-policies/editorial-policies}
%
%\bigskip\noindent
%BMC journals: \url{https://www.biomedcentral.com/getpublished/editorial-policies}
%\end{flushleft}
\appendix 

\section{Some technical proofs}\label{app:proofs}

Proof of Lemma~\ref{lemma_psi}.

\begin{proof}
	First, observe that, for any fixed $\lambda$, we have 
	\begin{equation}\label{psi_lim_dir}
	\lim_{ \gamma \rightarrow \infty} \psi (\gamma, \lambda) = 1.
	\end{equation}
	Now, given $\lambda > 0$, observe that 
	\begin{equation}\label{deriv_psi}
	\frac{ \partial \psi}{ \partial \gamma} (\gamma, \lambda) = \frac{ (2\lambda - 1) \gamma^2  + \lambda }{ \sqrt{1+\gamma^2}(\gamma^2+\lambda)^2 }.
	\end{equation}
	\begin{itemize}
		%		\item[(a)] Se $\lambda = 0$, então
		%		\begin{equation*}
		%		\frac{ \partial \psi}{ \partial \gamma} (\gamma, \lambda) = \frac{ - 1 }{ \sqrt{1+\gamma^2}} < 0, \quad \forall \gamma,
		%		\end{equation*}
		%		provando que $\psi$ é decrescente em $\gamma$. Mais ainda, como
		%		\begin{equation*}
		%		\lim_{\gamma \rightarrow 0} \psi (\gamma, \lambda) = \lim_{\gamma \rightarrow 0} \dfrac{ \sqrt{1+\gamma^2}}{\gamma} = \infty,
		%		\end{equation*}
		%		concluímos a afirmação do item.
		\item[(a)] Consider a fixed $\lambda \in (0,1/2)$. Since the maximizers of $\psi(\gamma,\lambda)$, for $\gamma > 0$ must satisfy $\frac{ \partial \psi}{ \partial \gamma} (\gamma, \lambda) = 0$, from \eqref{deriv_psi}, we have 
		\begin{equation*}
		(2\lambda - 1) \gamma^2  + \lambda = 0 \quad \Rightarrow \quad \gamma_{\max} = \sqrt{ -\frac{ \lambda}{ 2\lambda-1}},
		\end{equation*}
		%por estarmos considerando apenas $\gamma \geq 0$. Observe que $\gamma_{\max} \in \R$ pois $2\lambda-1 < 0$, de modo que o termo sob a raiz é positivo. Mais ainda, pelo comportamento de $(2\lambda - 1) \gamma^2  + \lambda$ como função de $\gamma$, é fácil ver que $\frac{ \partial \psi}{ \partial \gamma} (\gamma, \lambda) > 0$ se $\gamma < \gamma_{\max}$ e $\frac{ \partial \psi}{ \partial \gamma} (\gamma, \lambda) < 0$ se $\gamma > \gamma_{\max}$. Desta forma, concluímos que $\psi$ é crescente para $0 \leq \gamma < \gamma_{\max}$ e em seguida decresce, caracterizando portanto $\gamma_{\max}$ como ponto de máximo. 
		Replacing $\gamma_{\max}$ at $\psi(\gamma,\lambda)$, we obtain 
		\begin{equation*}
		\psi(\gamma_{\max},\lambda) = \frac{1}{ 2 \sqrt{ \lambda - \lambda^2 }}.
		\end{equation*}
		Furthermore, since $\frac{ \partial \psi}{ \partial \gamma} (\gamma, \lambda) <0$, for $\gamma \in (\gamma_{\max},+\infty)$ and $\frac{ \partial \psi}{ \partial \gamma} (\gamma, \lambda) > 0$, for $\gamma \in (0,\gamma_{\max})$, with $\psi(0,\lambda) = 0$, we conclude that $\gamma_{\max}$ is the unique maximizer.
%		Ademais, comentamos que $\psi(\gamma_{\max},\lambda) > 1$, o que é superior aos valores nos extremos para $\gamma \in [0, \infty]$, pois $\lim_{ \gamma \rightarrow 0} \psi (\gamma, \lambda) = 0$ e $\lim_{ \gamma \rightarrow \infty} \psi (\gamma, \lambda) = 1$ (por (\ref{psi_lim_dir})).
		\item[(b)] Now, for a fixed $\lambda \in [1/2,+\infty)$, in view of \eqref{deriv_psi}, it follows that $\frac{ \partial \psi}{ \partial \gamma} (\gamma, \lambda) > 0$ for $\gamma \geq 0$. Hence, $\psi$ is increasing for $\gamma \in [0, +\infty)$, with $\psi(0,\lambda)=0$. Then, from \eqref{psi_lim_dir}, we conclude that $\psi (\gamma, \lambda) \leq 1$, for all $\gamma \geq 0$.
	\end{itemize}
\end{proof}

%%=============================================%%
%% For submissions to Nature Portfolio Journals %%
%% please use the heading ``Extended Data''.   %%
%%=============================================%%

%%=============================================================%%
%% Sample for another appendix section			       %%
%%=============================================================%%

%% \section{Example of another appendix section}\label{secA2}%
%% Appendices may be used for helpful, supporting or essential material that would otherwise 
%% clutter, break up or be distracting to the text. Appendices can consist of sections, figures, 
%% tables and equations etc.

%%===========================================================================================%%
%% If you are submitting to one of the Nature Portfolio journals, using the eJP submission   %%
%% system, please include the references within the manuscript file itself. You may do this  %%
%% by copying the reference list from your .bbl file, paste it into the main manuscript .tex %%
%% file, and delete the associated \verb+\bibliography+ commands.                            %%
%%===========================================================================================%%

%%\section*{Data availability statement}
%%The datasets generated during and/or analyzed during the current study are available from the corresponding author on reasonable request. 
%%
%% \section*{Conflict of interest}
%% The authors declare that they have no conflict of interest.

\bibliographystyle{plain}
\bibliography{refs}% common bib file
%% if required, the content of .bbl file can be included here once bbl is generated
%%\input sn-article.bbl

%% Default %%
%%\input sn-sample-bib.tex%

\end{document}